\documentclass[10pt,a4paper]{amsart}
\usepackage{hyperref}
\usepackage{mathrsfs}
\usepackage{amsmath}
\usepackage{amssymb}
\usepackage[all]{xy}
\usepackage{latexsym}
\usepackage[utf8]{inputenc}
\usepackage[T1]{fontenc}
\usepackage{pigpen}

\title[Algebraic infinite delooping]{Algebraic infinite delooping and derived 
destabilization}
\author[G. Powell]{Geoffrey Powell}
\address{LAREMA, CNRS, Université d'Angers, Université 
Bretagne Loire, 2 Bd lavoisier, 49045 Angers, France}
\email{Geoffrey.Powell@math.cnrs.fr}
\date{}
\thanks{This work was partially supported by the 
project {\em Nouvelle 
Équipe},  convention No. 2013-10203/10204 between the Région des Pays de la 
Loire and the Université d'Angers.}
\newtheorem{THM}{Theorem}
\newtheorem{COR}[THM]{Corollary}
\newtheorem{thm}{Theorem}[section]
\newtheorem{prop}[thm]{Proposition}
\newtheorem{cor}[thm]{Corollary}
\newtheorem{lem}[thm]{Lemma}
\theoremstyle{definition}
\newtheorem{defn}[thm]{Definition}
\newtheorem{exam}[thm]{Example}

\theoremstyle{remark}
\newtheorem{rem}[thm]{Remark}
\newtheorem{nota}[thm]{Notation}

\renewcommand{\phi}{\varphi}
\renewcommand{\epsilon}{\varepsilon}
\newcommand{\po}{\ar@{}[dr]|(.7){\text{\pigpenfont R}}}
\newcommand{\pb}{\ar@{}[dr]|(.3){\text{\pigpenfont J}}}

\newcommand{\dash}{\mbox{-}}

\newcommand{\unst}{\mathscr{U}}
\newcommand{\nat}{\mathbb{N}}
\newcommand{\zed}{\mathbb{Z}}
\newcommand{\field}{\mathbb{F}}
\newcommand{\cala}{\mathscr{A}}
\newcommand{\calc}{\mathcal{C}}

\newcommand{\indec}{\mathbf{q}}
\newcommand{\triv}{\mathbf{triv}}
\newcommand{\qmbig}{\widetilde{\mathcal{QM}}}
\newcommand{\qm}{\mathcal{QM}}
\newcommand{\qmgr}{\qm^{\mathrm{gr}}}
\newcommand{\sing}{\mathscr{R}}
\newcommand{\amod}{\mathcal{M}}
\newcommand{\amodc}{\underline{\amod}}
\newcommand{\amodgr}{\amod^{\mathrm{gr}}}
\newcommand{\qsusp}{\overline{\Sigma}}
\newcommand{\lder}{\mathbb{L}}
\newcommand{\qcx}{\mathfrak{C}_\indec}
\newcommand{\redqcx}{\overline{\qcx}}
\newcommand{\ch}{\mathrm{Ch}}
\newcommand{\untor}{\mathrm{Untor}^R}
\newcommand{\anr}[1][1]{a_{#1}R \dash \mathrm{mod}}
\newcommand{\anar}[1][1]{a_{#1}\cala^\circ R \dash\mathrm{mod}}
\newcommand{\rcx}{\mathfrak{R}}
\newcommand{\imQ}[1][M]{\mathrm{im}(Q_0)_{#1}}
\newcommand{\lgth}[1]{^{[#1]}}
\newcommand{\filt}{\mathfrak{f}}

\newcommand{\hilb}{\mathsf{H}}
\newcommand{\versch}{\tilde{\unst}_* ^\field}
\newcommand{\vconn}{{\unst}_* ^\field}

\newcommand{\vertic}{\mathfrak{v}}
\newcommand{\stein}{\mathscr{L}}
\newcommand{\hecke}{\mathscr{H}}
\newcommand{\kz}[1][s]{\mathfrak{L}_{#1}}
\newcommand{\gr}{\mathfrak{gr}}
\begin{document}

\begin{abstract}
Working over the prime field of characteristic two, consequences of the Koszul 
duality between the Steenrod algebra $\cala$ and the (big) Dyer-Lashof algebra are 
studied, with an emphasis on the interplay between instability for the Steenrod algebra action 
and that for the Dyer-Lashof operations. The central algebraic framework is the category 
$\qmgr$ of length-graded modules over the Steenrod algebra equipped with an unstable action of the 
Dyer-Lashof algebra, with compatibility via the Nishida relations.

A first ingredient is  a functor from  modules over the Steenrod algebra, 
$\amod$, to $\qmgr$, that arose in the work of Kuhn and McCarty on the homology of infinite loop spaces.  
This functor is given in terms of derived functors of destabilization from the category $\amod$ of modules over the Steenrod 
algebra to unstable modules, enriched by taking into account the action of Dyer-Lashof operations.

A second ingredient is the derived functors of the Dyer-Lashof indecomposables 
functor 
$\indec : \qmgr \rightarrow \amodgr$
 to length-graded modules over the Steenrod algebra. These are related to functors used by Miller
 in his study of a spectral sequence to calculate the homology of an infinite delooping. 
An important fact is that these functors can be calculated as the homology of 
an explicit Koszul complex with terms expressed as certain Steinberg functors. 
The latter are quadratic dual to the more familiar Singer functors.

By exploiting the explicit complex built from the Singer functors which 
calculates the derived functors of destabilization,
 Koszul duality leads to an algebraic infinite delooping spectral sequence. This 
is conceptually similar to Miller's spectral sequence,
 but there seems to be no direct relationship. 

The spectral sequence sheds light on the relationship between unstable modules 
over the Steenrod algebra and all $\cala$-modules.
\end{abstract}

\maketitle

\section{Introduction}

To motivate the study of the rich algebraic 
structures which intervene, consider a spectrum $X$ and its associated infinite 
loop space $\Omega^\infty 
X$ and their respective mod-$2$ homologies (denoted here 
simply by $H_*(X)$ and $H_* (\Omega^\infty X)$). 
The homology of the spectrum, $H_* (X)$, is an $\cala$-module, where the mod 
$2$ 
Steenrod algebra $\cala$ acts on the right. 
However, unlike the homology of a space, it is not an unstable 
$\cala$-module in general. On the other hand, $H_*(\Omega^\infty X)$ is an 
unstable 
$\cala$-module and there is much more structure: it is a bicommutative Hopf 
algebra and 
is equipped with an action of the Dyer-Lashof algebra and all these structures 
are 
compatible. Here the focus is upon the residual additive structure after 
passage to the algebra indecomposables $QH_* (\Omega^\infty X)$. 

Given $H_* (X)$ as an $\cala$-module, there are various strategies available 
for calculating $H_* (\Omega^\infty X)$, at least when $X$ is connected. For 
instance, Kuhn and McCarty \cite{KMcC} use the 
spectral sequence of the associated Goodwillie-Arone tower;  Haugseng and 
Miller 
\cite{MH} use the 
spectral sequence associated to a cosimplicial resolution constructed from 
Eilenberg-MacLane spectra.

In their work, Kuhn and McCarty give an algebraic approximation to their spectral sequence 
with 
$E^\infty$-page expressed in terms of the derived functors of destabilization;  
an important fact is that the Dyer-Lashof action is visible on the 
$E^\infty$-page. (In this homological setting, the  destabilization functor $\Omega^\infty : \amod 
\rightarrow \unst$ is the right adjoint to the inclusion of the category $\unst$ 
of unstable 
modules in $\amod$, the category of $\cala$-modules. This functor is left exact 
and has 
non-trivial right derived functors $\Omega^\infty_i$.)

From a purely algebraic viewpoint, the fundamental algebraic object appearing 
in 
\cite{KMcC} is, 
for $M \in \amod$, the length-graded $\cala$-module 
 \[
H_0 \rcx M :=   \bigoplus_i \Sigma^{-1} \Omega^\infty_i \Sigma^{1-i}M, 
 \]
 equipped with an unstable action of the (big) Dyer-Lashof algebra that is 
compatible with the $\cala$-action via the Nishida relations. The length 
grading 
reflects the fact that the (big) Dyer-Lashof algebra is a homogeneous 
quadratic algebra, hence has a length grading.

Such objects form a category $\qmgr$ and the above can be considered as a 
functor 
from $\amod $ to $\qmgr$. (Forgetting the length grading gives the category 
$\qm$ used 
by Kuhn and McCarty.) As  the notation suggests, $H_0\rcx M$ is the zeroth 
homology of a functorial chain complex $ \rcx M$ taking values in $\qmgr$. 
 A secondary purpose of this paper is to underline the importance of 
the functor 
\[
 H_0 \rcx : \amod \rightarrow \qmgr 
\]
and give its fundamental properties (see Sections \ref{sect:destab} and \ref{sect:H0R_susp}).

The object appearing in Kuhn and McCarty's algebraic approximation is not 
$H_0 \rcx M$ (for $M= H_* X$) but $\qsusp H_0\rcx \Sigma^{-1} M$, where 
$\qsusp : \qmgr \rightarrow \qmgr$ is the left adjoint to the desuspension 
functor 
$\Sigma^{-1} : \qmgr \rightarrow \qmgr$. This exhibits the interplay between 
the notions 
of instability with respect to the $\cala$-module structure and with respect to 
the Dyer-Lashof action. 
For instance, $H_0 \rcx M$ is not necessarily $\cala$-unstable but is 
Dyer-Lashof unstable. The suspension 
$\Sigma H_0 \rcx M $ is $\cala$-unstable, but need not be Dyer-Lashof unstable; 
the functor $\qsusp$ corrects this. 
 It is also interesting to understand the relationship between $H_0 \rcx M$ and 
$H_0 \rcx \Sigma M$; this is explained in Corollary 
\ref{cor:delooping_H0RM}, generalizing a classical result for unstable modules 
over the Steenrod algebra.

To introduce the second ingredient, let us return to topology and the attempt 
to 
recover $H_*(X)$ from 
 $H_* (\Omega^\infty X)$ by killing  the extraneous 
structure. This strategy was carried out by 
Miller, who constructed an infinite delooping spectral sequence \cite{Miller}.
The first step passes to the algebra indecomposables $QH_* (\Omega^\infty X)$ 
and the second to the indecomposables for the residual Dyer-Lashof action; in 
the 
spectral sequence, the higher derived functors of these appear. The restriction axiom 
for the action 
of Dyer-Lashof operations $Q_0 x =x^2$ implies that $QH_* (\Omega^\infty X)$ is 
the desuspension of an object of $\qm$, 
which leads Miller to work with a slightly different category.

From the current viewpoint, the basic functor to consider is the Dyer-Lashof 
indecomposables functor 
\[
 \indec : \qmgr \rightarrow \amodgr
\]
that is left adjoint to the trivial Dyer-Lashof action functor (here $\amodgr$ 
is the category of length-graded $\cala$-modules). This has left derived 
functors (in the sense of relative
homological algebra), $\lder_i \indec$, which are of independent interest. 
These 
are related to the functors $\untor_* (\field, -)$
 used by Miller, who identified some key properties.

Restricted to $\amod$, the functors $\lder_n \indec$ ($n \in \nat$) turns out to be 
non-trivial only in length grading $n$; this is a manifestation
of the Koszul property, à la Priddy \cite{Priddy}. Moreover, the associated 
functor $\stein_n : \amod \rightarrow \amod$ (termed
here the Steinberg functor), is exact. This functor is quadratic dual to the 
important Singer functor $\sing_n : \amod \rightarrow \amod$ 
 which arises as a component of the left adjoint to the forgetful functor $\qm 
\rightarrow \amod$. The Steinberg functors do not 
seem to have been studied in full generality, but arise for example implicitly 
in the work 
of Kuhn on the Whitehead conjecture \cite{Kuhn:Whitehead}. 

For $N \in \qmgr$, with length-grading components $N\lgth{i}$ ($i \in \nat$), 
in fixed length grading $n \in \nat$, there 
is a  Koszul complex $\kz[n] N$ in 
$\amod$:
 \[
  \stein_n N\lgth{0}
  \stackrel{d^\stein_n }{\rightarrow}
  \stein_{n-1} N\lgth{1}
  \stackrel{d^\stein_{n-1} }{\rightarrow}
\stein_{n-2} N\lgth{2}
\rightarrow 
\ldots 
\rightarrow 
\stein_0 N \lgth{n} = N\lgth{n},
  \]
where $\stein_{i} N\lgth{n-i}$ is placed in homological degree $i$.
The differential is defined using the action of Dyer-Lashof operations on $N$.
A related complex occurs in the work of Miller \cite{Miller}.

  By Corollary \ref{cor:vanishing_lder_indec}, the $i$th homology of $\kz[n]N$ is 
$(\lder_i \indec N)\lgth{n}$.  Since 
the Steinberg functors are relatively well understood as functors on $\amod$, 
this leads 
to a good understanding of the derived functors $\lder_i \indec$.

Combining the Kuhn-McCarty algebraic approximation with the existence of 
Miller's 
infinite delooping spectral sequence leads to the following question: is 
there a natural spectral sequence that recovers $M \in \amod$ from the input 
$H_0 \rcx M $? For convergence reasons, it is only reasonable to ask this 
question when $M$ is 
bounded below, but one does not necessarily want to insist that $M$ be 
$(-1)$-connected. In this generality the answer is {\em no}, as can be seen by  
considering the module $M = \Sigma^{-2} \cala^*$ ($\cala^*$ the dual Steenrod 
algebra), since $H_0 \rcx \Sigma^{-2} \cala^*$ is zero.
 
This does not preclude the existence of such a spectral sequence to recover $M$ 
from $H_0 \rcx M$ 
when $M$ is $0$-connected. For example, one could use  
 the filtration of an  $\cala$-module given by Miller in \cite{Miller}. However, the author
knows of no approach which leads to such a spectral sequence with an explicitly 
identified $E^2$-page.

On the other hand, the Koszul duality between Steenrod and Dyer-Lashof actions {\em does} lead 
to a spectral sequence that gives important information on the structure of a 
bounded-below $\cala$-module.
This is the subject of Section \ref{sect:ss}, which contains the following 
result:

\begin{THM}
 Let  $M \in \amod$ be bounded-below.
\begin{enumerate}
 \item 
There is a natural second quadrant homological spectral sequence  $(E^r_{s,t}, 
d^r)$
 with $d_r$ of $(s,t)$-bidegree $(r-1, -r)$ and $E^2$-page: 
\[
 E^2_{s,t} =
\lder_t\indec \Big(
H_0 \rcx \Sigma^{-s}M
\Big)\lgth{-s}.
\]
\item 
The spectral sequence converges strongly to  $M$ with associated increasing 
filtration
$\vertic_t M \subset M$, $t \in \nat$.
\item 
There is a natural isomorphism
\[
 E^2_{-t, t} 
 \cong 
 \ker 
 \big \{
 \stein_t \Sigma^{-1} \Omega^\infty \Sigma^{1+t}M 
 \rightarrow 
 \stein_{t-1} \Sigma^{-1} \Omega^{\infty}_1 \Sigma^{t}M 
 \big \}.
\]
 \item 
The $t=0$ edge morphism identifies with the natural inclusion
\[ 
\vertic_0 M \cong 
\Sigma^{-1} \Omega^{\infty} \Sigma M  
\hookrightarrow 
 M
\]
in $(s,t)$-degree $(0,0)$ and is zero elsewhere.
\item
$E^2_{s,t}=0$ for $t>-s$, so that the spectral sequence 
is concentrated below the anti-diagonal and there is a second edge homomorphism:
\[
 E^\infty_{-t, t} \cong \vertic_t M / 
\vertic_{t-1} M
\hookrightarrow 
 E^2_{-t, t}.
 \]
\item 
$E^\infty_{s,t}=0$ if $s \neq -t$ and $E^\infty_{-t, t} \cong \vertic_t M / 
\vertic_{t-1} M$.
\end{enumerate}
\end{THM}

Some comments are in order. The length zero part of $H_0 \rcx \Sigma^l M $ is 
$\Sigma^{-1} \Omega^{\infty} \Sigma^{1+l}M$. 
Since any bounded-below $\cala$-module $M$ can be recovered as 
\[
 M \cong \mathrm{colim}_d \Sigma^{-d} \Omega^\infty \Sigma^d M
\]
one must check that the spectral sequence really does give an effective means of 
recovering $M$. This is the case, 
since the contribution from $H_0 \rcx \Sigma^l M $ has connectivity which 
increases exponentially with $l$ (for $l$ sufficiently large).

 Some of the consequences, including connectivity 
results, are summarized in the following:

\begin{COR}
 Let $M$ be a bounded-below $\cala$-module. Then $M$ admits a natural, 
exhaustive filtration in $\amod$:
\[
 0 
\subset 
\vertic_0 M 
\subset 
\vertic_1 M 
\subset 
\ldots 
\subset 
\vertic_t M 
\subset 
\ldots 
\subset M
\]
such that 
\begin{enumerate}
 \item 
$\vertic_0 M = \Sigma^{-1}\Omega^\infty \Sigma M$;
\item 
for $t \in \nat$, 
$$\vertic_t M / \vertic_{t-1}M
\subset 
\ker 
 \big \{
 \stein_t \Sigma^{-1} \Omega^\infty \Sigma^{1+t}M 
 \rightarrow 
 \stein_{t-1} \Sigma^{-1} \Omega^{\infty}_1 \Sigma^{t}M 
 \big \};
$$
\item 
if $M$ is $c$-connected, then $ \vertic_t M / \vertic_{t-1}M$ is at least 
$\big(2^t (d(M)+1) - (t+2)\big)$-connected, where 
 $d(M)= \sup \{ (c+t+1), -1 \}$;
\item 
the module $\vertic_t M \subset M$ admits a finite filtration such that the 
associated graded satisfies:
 \[
  \Sigma^{2^t} \gr\big( \vertic_t M\big) \in \unst
 \]
is unstable.
\end{enumerate}
Moreover, there is a natural morphism of filtered objects relating the 
respective filtrations for $M$ and $\Sigma M$:
\[
 \vertic_t M \hookrightarrow \Sigma^{-1} \vertic_t \Sigma M
\]
which, for $t=0$, is the natural inclusion:
\[
 \Sigma^{-1} \Omega^\infty \Sigma M
\hookrightarrow 
\Sigma^{-2} \Omega^\infty \Sigma^2 M.
\]
\end{COR}

In the case $M= \Sigma^n \cala^*$, for $n \in \zed$, one obtains a skewed length 
filtration of the (dual) Steenrod algebra (see Section \ref{subsect:Sigma_n_cala}). 
This should be compared with Miller's results, which exhibits the length 
filtration of $\cala/ \cala Sq^1$ arising from 
his spectral sequence. It is of interest that an analogous calculation occurs in 
the work of Haugseng and Miller \cite{MH}, which is `dual', being related to the Dyer-Lashof algebra. 

Clearly, if $M = \Sigma^{-1}N $ where $N$ is unstable, one has $\vertic_0 M = M$ 
and the filtration is constant. The case 
$M = \Sigma^{-2}N $ ($N$ as above), already shows up non-trivial behaviour and 
passage to further desuspensions 
increases complexity (see Section \ref{subsect:Sigma_-2_unstable}).

\section{Preliminaries}
\label{sect:prelim}

\subsection{Recollections on $\cala$-modules}

Throughout $\field$ denotes the prime field of characteristic two and $\cala$ 
the mod-$2$ Steenrod algebra.
Unless otherwise stated, all $\cala$-modules are {\em homological}, so that 
operations act on the right 
and $Sq^i$ has homological degree $-i$. 

Here and in Section \ref{sect:dl}, the conventions of Kuhn and McCarty 
\cite{KMcC} are followed; for unstable modules in the cohomological framework 
(ie. with $\cala$ acting on the left), see \cite{schwartz_book}. 

Duality permits the passage between the homological and cohomological settings, 
restricting to an equivalence
 of categories on objects of finite type. In particular, if $M$ is a 
cohomological $\cala$-module, 
$M^*$ denotes the dual homological $\cala$-module.

\begin{nota} (Cf. \cite{KMcC}.)
Denote by
\begin{enumerate}
 \item 
$\amod$ the category of locally-finite $\cala$-modules and $\amodc \subset  
\amod$ the full subcategory of modules that are bounded 
below; 
\item
$\Sigma : \amod \rightarrow \amod$ the suspension functor;
\item 
$\Phi : \amod \rightarrow \amod$ the Frobenius functor which doubles degrees; 
\item 
$\unst \subset \amodc \subset \amod$ the full subcategory of unstable modules;
\item 
$\Omega^\infty : \amod \rightarrow \unst$ the right adjoint to the inclusion 
$\unst \hookrightarrow \amod$, with right derived functors $\Omega^\infty_i$, 
$i 
\geq 0$.
\end{enumerate}
\end{nota}

\begin{nota}
For $M$ a homologically $\zed$-graded object and $n \in \zed$, let $M_n$ denote 
 
the degree $n$ component of $M$. 
\end{nota}

\begin{defn}
\label{def:conn}
A $\zed$-graded object $M$ is $c$-connected if $M_n =0$ for $n \leq c$. In 
particular $M$ is $(-1)$-connected if and only if it is concentrated in 
non-negative degrees.
\end{defn}

\begin{exam}
 \label{exam:dual_Steenrod}
 The dual Steenrod algebra $\cala^*$ is an $(-1)$-connected  object of $\amod$ 
hence, for any $n \in \zed$, $\Sigma^n 
\cala^*$ lies in  $\amodc$. The higher derived functors of $\Omega^\infty$ 
vanish on these objects:
 \[
  \Omega^\infty_i \Sigma^n \cala^* =0
 \]
for $i >0$. For $i=0$, one has $\Omega^\infty \Sigma^n \cala^* = F(n) ^* \in 
\unst$, where $F(n)$ is the free (cohomological) unstable module on a generator 
of 
degree $n$, in particular is zero for $n <0$.
\end{exam}

The categories $\amod$, $\amodc$ and $\unst$ are abelian and the functors 
$\Sigma$ and $\Phi$ are exact. The following Proposition is  the homological 
version of a standard result for cohomological unstable modules:

\begin{prop}
\label{prop:Omega_es}
The functors $\Sigma$, $\Phi$ restrict to endofunctors of $\unst$ and 
the top Steenrod operation induces a natural 
transformation $Sq_0 : 1_\unst \rightarrow \Phi $. 

The suspension $\Sigma$  has right adjoint $\Omega : \unst \rightarrow \unst$ 
which fits into  the 
natural exact sequence 
\begin{eqnarray}
\label{eqn:Omega_es}
 0
\rightarrow 
\Sigma \Omega M 
\rightarrow 
M 
\stackrel{Sq_0}{\rightarrow} 
\Phi M 
\rightarrow 
\Sigma \Omega_1 M 
\rightarrow 
0,
\end{eqnarray}
for $M \in \unst$,  where $\Omega_1$ is the first right derived functor of 
$\Omega$ and the higher derived functors vanish.
\end{prop}

\begin{defn}
\label{def:q0}
 For $M \in \unst$, let $q_0 : \Phi \Sigma \Omega M \rightarrow \Sigma \Omega_1 
M$ denote the natural transformation 
 given by the composite 
 \[
  \Phi \Sigma \Omega M 
  \hookrightarrow 
  \Phi M 
  \rightarrow 
  \Sigma \Omega_1 M
 \]
of the morphisms from the exact sequence (\ref{eqn:Omega_es}) of Proposition 
\ref{prop:Omega_es}.
\end{defn}

\begin{rem}
 The notation $q_0 : \Phi \Sigma \Omega M \rightarrow \Sigma \Omega_1 M$ 
indicates a relationship with the Dyer-Lashof operation $Q_0$ (cf.  
Section \ref{sect:dl}).
\end{rem}

The derived functors $\Omega^\infty_s$ ($s \in \nat$) and $\Omega_i$ ($i \in \{ 
0, 1\}$) are related by a short 
exact sequence derived from the Grothendieck spectral sequence for the composite 
of functors. 
Short exact sequences of this form go back 
to the work of Singer (cf. \cite{Singer} for example).

\begin{prop}
\label{prop:Omega_infty_desusp}
 For $M \in \amod$ and $s \in \nat$, there is a natural short exact sequence in 
$\amod$:
 \[
  0
  \rightarrow 
  \Omega_1 \Omega^\infty _s M
  \rightarrow 
  \Omega^\infty_{s+1} \Sigma^{-1} M
  \rightarrow 
  \Omega \Omega^\infty_{s+1} M
  \rightarrow 
  0.
 \]
\end{prop}

\subsection{A delooping spectral sequence}
\label{subsect:ss_single_loop}

As a warm-up to considering the passage between unstable $\cala$-modules and 
$\cala$-modules, consider delooping an unstable module of the form $\Omega M$.  

\begin{defn}
\label{def:filt_Phi}
 For $M \in \unst$, let $\filt_i M$ denote the natural 
increasing filtration of $M$ defined 
 by $\filt_{-1} M = 0$ and
 \[
  \filt_i M := \ker \{ M \stackrel{Sq_0^i}{\rightarrow} \Phi^{i+1} M \}
 \]
 for $i \in \nat$.
\end{defn}

\begin{rem}
By definition, $\filt_{-1} M = 0$ and $\filt_0 M = \Sigma \Omega M$. 
Moreover, $\bigcup_i \filt_i M = \overline{M}$, the submodule of elements in 
positive degree. Hence the filtration is exhaustive if and only if $M$ is  
$0$-connected.
\end{rem}

The Hilbert series of an $\nat$-graded $\field$-vector space $V$ is defined by 
$\hilb_V (t) := \sum_{n \geq 0} \dim V_n t^n$. One has  
$\hilb _{\Phi V}(t) = \hilb _V (t^2)$.

\begin{rem}
 If $M$ is $0$-connected, then the Hilbert series of $M$ is determined 
recursively by those of $\Omega M$ and $\Omega_1 M$ by using the Euler 
characteristic of the exact sequence (\ref{eqn:Omega_es}) of Proposition 
\ref{prop:Omega_es}, which gives the identity:
 \[
  \hilb _M (t) - \hilb_M (t^2) 
  = 
  t (\hilb_{\Omega M} (t) - \hilb_{\Omega_1 M} (t)).
 \]
This observation can be refined using the spectral sequence associated 
to the filtration of Definition \ref{def:filt_Phi}.
\end{rem}

The  filtration $ \filt_i M$ is based upon the action of $Sq_0$ upon 
$M$, that is $M$ considered as an $\nat$-graded vector space equipped with a 
{\em Verschiebung}. These 
objects form an abelian category $\versch$ (using the notation of \cite[Section 
1.1.1]{GLM}); the full subcategory of $0$-connected objects is denoted 
$\vconn$. The category $\vconn$ 
splits as a product of categories indexed by the prime to $2$ part of the 
degree; 
each of the factors is equivalent to the category of graded modules 
concentrated in negative degree
 over the graded polynomial ring $\field [x]$ with $x$ in degree one, hence the 
homological algebra of these categories is well understood. 

Consider $M \in \vconn$ of finite type. There is a minimal injective 
resolution in $\vconn$ of the form 
\[
0\rightarrow
 M \rightarrow I_0 
 \rightarrow I_1 
 \rightarrow 0
\]
and, borrowing the notation used for $\unst$, the socle of $I_0$ is isomorphic 
to $\Sigma \Omega M$ and the socle to $I_1$ to $\Sigma \Omega_1 M$.

The natural filtration of Definition \ref{def:filt_Phi} provides a filtered 
complex $\filt_* I_\bullet$ and hence an associated spectral sequence, which 
has a particularly simple form. The $E^0$-page is concentrated on  
 the diagonals of total (cohomological) degree $0$ and $1$, the differential 
$d_0$ is trivial (due to the hypothesis of minimality) and the differential 
$d_1$ is induced by the natural transformation $q_0$ of Definition \ref{def:q0}.
 
The hypothesis that the resolution be minimal served only to 
deduce that the differential $d_0$ is trivial; in the general case,  the 
spectral 
sequence identifies as above from the $E^1$-page. This allows the result to be 
made more precise when working with $\unst$. In the following, {\em 
homological} indexing is used, so that $I_1$ is placed in degree $-1$.

\begin{thm}
\label{thm:delooping_ss}
Let $M \in \unst$ be $0$-connected and of finite type. Then there is a strongly 
convergent homological spectral sequence of unstable modules that is 
concentrated in the fourth quadrant and satisfies the following properties.
\begin{enumerate}
 \item 
 The non-zero terms of the $E^1$-page are: 
\begin{eqnarray*}
 E^1_{p,-p} &\cong& \Phi^p \Sigma \Omega M \\
 E^1_{p,-p-1} &\cong& \Phi^p \Sigma \Omega_1 M
\end{eqnarray*}
for $p \geq 0$, with $d_1: E^1_{p,-p} \rightarrow E^1_{p-1, -p}$ zero for $p=0$ 
and 
$\Phi^{p-1} q_0$, for $p>0$. 
\item 
The $E^2$-page is given by 
\begin{eqnarray*}
 E^2_{0,0} & \cong & \Sigma \Omega M \\
 E^2_{p,-p} &\cong& \Phi^{p-1} \ker q_0 \\
 E^2_{p,-p-1} &\cong& \Phi^p \mathrm{coker}\  q_0
\end{eqnarray*}
for $p > 0$.
\item 
For $r \geq 1$, the differential $d^r$ is determined by $d^r : E^r_{r,-r} 
\rightarrow E^r_{0, -1}$  for $n \in \nat$ via:
$$
E^r _{r+n, -r +n}  \rightarrow   E^r_{n, n-1} = \Phi^n (E^r_{r,-r} 
\stackrel{d^r}{\rightarrow} E^r_{0, -1}) .
$$
\item 
The  $E^\infty$-page is concentrated on the anti-diagonal, 
with $E^\infty_{p, -p} \cong \filt_p M / \filt_{p-1}M$.
\end{enumerate}
\end{thm}

\begin{proof}
 (Indications.) 
Because of the finite-type hypothesis, one can dualize and work with 
cohomological 
unstable modules, where the argument may be more familiar. In this case, $M$ 
admits a projective resolution by $0$-connected projectives of finite type. 
Since unstable projectives are reduced, this is also a projective resolution in 
vector spaces with a Frobenius. The spectral sequence is then constructed as 
indicated in the preceding discussion. 
\end{proof}

\begin{rem}
\ 
\begin{enumerate}
 \item 
 Theorem \ref{thm:delooping_ss} provides a spectral sequence calculating $M$ 
with $E^2$-page expressed functorially in terms of $(\Sigma \Omega M, \Sigma 
\Omega_1 M , q_0)$. Clearly convergence requires that $M$ be $0$-connected.
\item 
The spectral sequence gives a systematic way of encoding the information 
contained in the exact sequence (\ref{eqn:Omega_es}) of Proposition 
\ref{prop:Omega_es}.
 \end{enumerate}
\end{rem}

\section{Dyer-Lashof actions}
\label{sect:dl}

This section introduces the algebraic categories which underlie the 
constructions of the paper, in particular the category $\qm$
 of modules equipped with actions of both the Steenrod algebra and Dyer-Lashof 
operations (satisfying additional conditions).

\subsection{Introducing the Dyer-Lashof operations}

As in \cite{KMcC}, Dyer-Lashof operations $Q^i$ are indexed by $i \in \zed$, 
with $|Q^i|=i$. (This should be contrasted with \cite{Miller}, where 
no negative Dyer-Lashof operations are required.) The lower indexing of 
Dyer-Lashof operations 
is useful:

\begin{nota}
Suppose that the  operations $Q^i$ act on the
$\zed$-graded module $M$; for $a \in \zed$  define $Q_a : \Sigma^a \Phi M 
\rightarrow M$ by 
$Q_a x := Q^{|x|+a} x$ (so that $Q_a$ is a linear map of degree zero).
\end{nota}

The $\cala$-modules with an action of Dyer-Lashof operations considered here 
satisfy: 
\begin{enumerate}
 \item 
{\bf Adem relations:}
\[
 Q^r Q^s = \sum _i \binom{i-s-1}{2i-r} Q^{r+s-i}Q^i; 
\]
\item 
{\bf Nishida relations:}
\[
 (Q^s x) Sq^r = \sum _i \binom{s-r}{r-2i} Q^{s-r+i}(xSq^i);
\]
\item
{\bf Dyer-Lashof instability:}
$Q_a$ acts trivially for $a <0$.
\end{enumerate}

\begin{rem}
\label{rem:big_DL}
\ 
\begin{enumerate}
 \item 
The instability condition implies that negative Dyer-Lashof operations act 
trivially on elements of non-negative degree. 
\item 
The Dyer-Lashof operations generate the {\em big Dyer-Lashof algebra}, namely 
the 
homogeneous quadratic algebra generated by 
$\{ Q^i | i \in \zed \}$ subject to the Adem relations.  The usual Dyer-Lashof 
algebra (as used for example in \cite{Miller}) is the quotient by the ideal 
generated by terms 
of negative excess.
\item 
The big Dyer-Lashof algebra is bigraded, with internal grading 
and length grading. The length grading plays a fundamental rôle here.
\end{enumerate}
\end{rem}

\begin{nota}
 Let 
 \begin{enumerate}
  \item 
  $\qmbig$ denote the category of modules over the big Dyer-Lashof algebra in 
$\amod$ which satisfy the Nishida relations; 
  \item
  let $\qm \subset \qmbig$ denote the full subcategory of objects satisfying 
the Dyer-Lashof instability condition.
\end{enumerate}
\end{nota}

\begin{rem}
 The category $\qm$ is of principal interest here. The suspension 
$\Sigma : \amod \rightarrow \amod$ induces a functor $\Sigma : \qmbig 
\rightarrow \qmbig$ (as for Steenrod operations, 
 Dyer-Lashof operations commute with the suspension). However, this does not 
preserve $\qm$, due to the relation
 \[
  Q_a \Sigma = \Sigma Q_{a+1}
 \]
for $a \in \zed$, which allows $Q_{-1}$ to act non-trivially on the 
suspension of an object of $\qm$. 
 
Here objects $M \in \qmbig$ such that $\Sigma^{-t} M \in \qm$ 
(for $t \in \nat$) arise occasionally; the full 
subcategory of such objects is denoted 
 $\Sigma^t \qm \subset \qmbig$. There is an increasing filtration 
 \[
  \qm \subset \Sigma \qm \subset \Sigma^2 \qm \subset \ldots \subset \qmbig 
 \]
(Cf. Proposition \ref{prop:qsusp} below). 
\end{rem}

\begin{lem}
\label{lem:qm_ab}
The category $\qm$ is abelian and the forgetful functor $\qm \rightarrow \amod$ 
is exact. Equipping an object of $\amod$ with trivial action by Dyer-Lashof 
operations defines an exact functor 
 $\triv : \amod \rightarrow \qm$.
\end{lem}

\begin{proof}
 Straightforward.
\end{proof}

\begin{prop} 
\label{prop:properties_sing}
\cite{KMcC}
 The forgetful functor $\qm \rightarrow \amod$ has an exact left adjoint 
$
 \sing : \amod \rightarrow \qm.
$ 
Moreover, $\sing$ takes values in the category of bigraded modules over the big 
Dyer-Lashof algebra, so that for $M \in \amod$
\[
 \sing M \cong \bigoplus_{s\geq 0} \sing_s M 
\]
where $s$ denotes the length grading and each $\sing_s M$ belongs to 
$\amod$, in particular $\sing_0 M =M$.

If $M \in \amod$ is $(d-1)$-connected, then $\sing_sM$ is $(2^s d 
-1)$-connected. In particular, 
if $M$ is $0$-connected, then so is $\sing M$ and, if moreover $M$ is of finite 
type, then so is $\sing M$. 
\end{prop}

\begin{proof}
 (Indications.) The functor $\sing$ is defined by forming the 
free module on the big Dyer-Lashof algebra
 and imposing the instability condition. The $\cala$-action is recovered from 
the Nishida relations. As in 
the proof of \cite[Lemma 4.19]{KMcC}, an explicit basis of $\sing_s M$ can be 
given in terms of allowable monomials, which 
shows that $\sing$ is exact. The final connectivity and finite-type statements 
follow easily from this (compare \cite[Lemma 4.10]{KMcC}).
\end{proof}

\begin{nota}
 Let $\indec : \qm \rightarrow \amod$ denote the Dyer-Lashof indecomposables 
functor, 
namely the left adjoint to  $\triv : \amod \rightarrow \qm$. 
\end{nota}

\begin{lem}
\label{lem:indec}
 \ 
\begin{enumerate}
 \item 
 The functor $\indec : \qm \rightarrow \amod$ is right exact.
\item 
The  composites of $\indec$ with $\sing : \amod \rightarrow \qm$ and  
$\triv : \amod 
\rightarrow \qm$ are both naturally equivalent to the identity functor of 
$\amod$.
\item 
For $N \in \qm$ and the adjunction counit $\omega_N : \sing N \rightarrow N$, 
the morphism $\indec \omega_N : N \twoheadrightarrow \indec N$ is the natural 
projection 
(adjunction counit for $\indec$). In particular, for $M \in \amod$ and $\mu_M : 
\sing \sing M \rightarrow \sing M$ the morphism $\omega _{\sing M}$, the 
morphism $\indec \mu_M : 
\sing M \rightarrow M$ is the projection  $\pi_M$ onto $M = \sing_0 M$. 
\end{enumerate}
\end{lem}

\begin{proof}
 Straightforward.
\end{proof}

\begin{prop}
\label{prop:Frobenius}
 The Frobenius functor $\Phi : \amod \rightarrow \amod$ extends to an exact 
functor 
\[
 \Phi : \qmbig \rightarrow \qmbig. 
\]
Restricted to $\qm$, the operation $Q_0$ induces a natural transformation 
$
Q_0 : \Phi \rightarrow 1_{\qm}.
$
\end{prop}

\begin{proof}
 The fact that $\Phi$ defines a functor $\Phi : \qmbig \rightarrow \qmbig$ is a 
straightforward calculation with the Adem and Nishida relations (analogous to 
the result for $Sq_0$ stated in Proposition 
\ref{prop:Omega_es}). In particular, the compatibility with the Adem relations 
follows from the congruence 
\[
 \binom{\alpha -1}{\beta} 
 \cong 
  \binom{2\alpha -1}{2\beta} 
\mod 2
  \]
and the compatibility with the Nishida relations from the congruence 
\[
 \binom{\alpha}{\beta} 
 \cong 
  \binom{2\alpha}{2\beta} 
\mod 2.
  \]

Restricting to $\qm$, consider the Adem relation for $Q^r Q_0 x$, where $Q_0$ 
acts via $Q^{|x|}$:
\[
 Q^r Q_0 x
 = 
 \sum _i \binom{i-|x|-1}{2i-r} Q^{r+|x|-i}Q^ix. 
\]
For the binomial coefficient to be non-trivial, we require $2i \geq r$, whereas 
instability for the Dyer-Lashof action 
implies that $Q^{r+|x|-i}Q^ix =0$ unless $r+ |x|-i \geq |x |+i$, namely $r \geq 
2i$. It follows that the Adem relation reduces to 
$Q^r Q_0 x=0$ if $r$ is odd and 
\[
 Q^{2i} Q_0 x = Q_ 0 Q^i x,
\]
as required.
\end{proof}

\begin{prop}
\label{prop:qsusp}
 The desuspension $\Sigma^{-1} : \amod \rightarrow \amod$ extends to an exact 
functor 
\[
 \Sigma^{-1} : \qm \rightarrow \qm.
\]
This admits a left adjoint $\qsusp : \qm \rightarrow \qm$ that fits into a 
natural exact sequence 
\[
 0
\rightarrow 
\Sigma^{-1} \qsusp_1 N 
\rightarrow 
\Phi N
\stackrel{Q_0}{\rightarrow}
N 
\rightarrow 
\Sigma^{-1} \qsusp N 
\rightarrow 
0
\]
for $N \in \qm$.
\end{prop}

\begin{proof}
 Analogous to the construction of the exact sequence of Proposition  
\ref{prop:Omega_es}. 
\end{proof}

\begin{rem}
The functor $\qsusp_1$ is the first left derived functor 
of $\qsusp$ in the sense of the relative homological algebra of Section 
\ref{sect:relhom}.
\end{rem}

\begin{nota}
 For $M \in \amod$, let 
$
\epsilon_M :  \sing M \rightarrow \Sigma^{-1} \sing \Sigma M
$
denote the natural transformation in $\qm$ induced by the natural inclusion 
$\Sigma M \hookrightarrow \sing \Sigma M$.
\end{nota}

\begin{cor}
\label{cor:qsusp}
 For $M \in \amod$ there is a natural isomorphism $\qsusp \sing M \cong \sing 
\Sigma M$ and the exact sequence of 
Proposition \ref{prop:qsusp} induces a natural short exact sequence 
\[
 0
\rightarrow 
\Phi \sing M
\stackrel{Q_0}{\rightarrow} 
\sing M
\stackrel{\epsilon_M}{\rightarrow} 
\Sigma^{-1} \sing \Sigma M
\rightarrow 
0
\]
in $\qm$. In particular, $Q_0$ induces a natural inclusion in $\amod$:
\[
 \Phi M 
 \hookrightarrow 
 \sing_1 M.
\]
\end{cor}

\begin{rem}
There is a natural transformation 
$
 \sing \Phi M 
\rightarrow 
\Phi \sing M
$ 
for $M \in \amod$,  which is induced by the natural inclusion $M 
\hookrightarrow 
\sing M$ in $\amod$. This is not injective in general, since the left hand side 
contains 
elements in odd degree.
\end{rem}

In a similar vein to Proposition \ref{prop:qsusp} is the following:

\begin{prop}
 \label{prop:Sq_0_DL-linear}
For $N \in \qmbig \cap \unst$, the linear map 
\[
 Sq_0: N \rightarrow \Phi N
\]
is a morphism of $\qmbig \cap \unst$.
\end{prop}

\begin{proof}
 This is proved by using the Nishida relations, using 
$\cala$-instability (without the hypothesis of Dyer-Lashof instability). The 
morphism $Sq_0$ is trivial on elements of odd degree, hence consider $(Q^s x) 
Sq_0$ where $s +|x|= 2j$, so that  $Sq_0$ acts via $Sq^j$. 
 The Nishida relation gives 
 \[
  (Q^s x) Sq_0 
  = \sum_i \binom{s-j}{j -2i} Q^{s-j + i} (xSq^i) 
 \]
where $2i \leq j$. By $\cala$-instability, the terms on the right hand side are 
trivial if $2i > |x|$, hence we may assume that  $2i \leq |x|$. The binomial 
coefficient is non-trivial only if $(s-j)\geq  (j-2i)$, equivalently if $2j -s 
\leq  2i$. Now $|x |= 2j - s$, hence non-trivial terms occur only when both $2i 
\leq |x|$ and $|x| \leq 2i$, hence only when $|x |= 2i$, so that $xSq^i = x 
Sq_0$.

For $|x|=2i$, the relation $s+|x| = 2j$ then implies $s-j = j -2i$ and $s =2n$ 
for some $n = j - i$ and one 
has $$(Q^{2n} x ) Sq_0 = Q^n (x Sq_0).$$ 
For $|x|$ odd, one has $(Q^s x)Sq_0 =0$. 
\end{proof}

\subsection{Length grading}
\label{subsect:length1}
As observed in  Remark \ref{rem:big_DL}, the big Dyer-Lashof algebra is 
bigraded when equipped with
  the length grading. Moreover the functor $\sing : 
\amod \rightarrow \qm$ takes values in the category of 
bigraded objects of $\qm$, namely those 
modules $N \in \qm$ equipped
with a length decomposition:
\[
 N \cong \bigoplus_{s \in \zed} N\lgth{s}
\]
such that $N \lgth{s} \in \amod$ and $Q_a : \Sigma^a \Phi N\lgth{s} \rightarrow 
N\lgth{s+1}$ 
(graded linear map, not $\cala$-linear in general).

\begin{nota}
\ 
\begin{enumerate}
 \item 
 Denote by $\qmgr$ the category of length-graded objects of $\qm$ and 
length-grading preserving morphisms, equipped
with the forgetful functor $\qmgr \rightarrow \qm$, $N \mapsto 
\bigoplus_{s 
\in \zed} N\lgth{s} $.
\item 
Denote by $\amodgr$ the category of length-graded objects of $\amod$ and 
length-grading preserving morphisms, equipped with the exact forgetful 
functor 
$\amodgr \rightarrow \amod$.
\item 
For $l \in \zed$, let $\cdot (l) : \qmgr \rightarrow \qmgr$  (respectively 
$\cdot (l) : \amodgr \rightarrow \amodgr$) denote the exact functor which 
increases length grading by $l$, so that 
$N(l)\lgth{s} = N\lgth{s-l}$.
\end{enumerate}
\end{nota}

\begin{rem}
 In applications here, the length grading is always bounded below, that is 
$N\lgth{s} =0$ 
for $s\ll 0$. Frequently the 
length grading will be defined only for $s \in \nat$, in which case it is 
extended by zero to negative degrees.
\end{rem}

\begin{prop}
\label{prop:qmgr}
\ 
\begin{enumerate}
\item 
The categories $\qmgr$ and $\amodgr$ are abelian and the forgetful functor 
$\qmgr\rightarrow \amodgr$ is exact.
\item 
The trivial action functor induces an exact functor $\triv : \amodgr 
\rightarrow \qmgr$.
\item 
For $l \in \zed$, the functor $\cdot (l) : \qmgr \rightarrow \qmgr$ is an 
equivalence of categories (respectively for $\amodgr$) and these equivalences 
are compatible via 
the forgetful and trivial action functors. 
\item 
The functor $\indec$ induces a functor $\indec : \qmgr \rightarrow \amodgr$.
\item 
The functor $\sing$ factorizes across a functor $\sing : \amod \rightarrow 
\qmgr$ which extends to 
a functor $\sing : \amodgr \rightarrow \qmgr$.
 \item 
The functors $\Sigma^{-1}$, $\qsusp$, $\Phi$ extend to functors on $\qmgr$. 
\item 
For $N \in \qmgr$, the morphism $Q_0$ defines a natural transformation in 
$\qmgr$:
\[
 \big(\Phi  N\big) (1) \stackrel{Q_0}{\rightarrow} N.
\]
\item 
For $M \in \amod$, the natural transformation $\epsilon_M :  \sing M 
\rightarrow 
\Sigma^{-1} \sing 
\Sigma M$ is defined in $\qmgr$ and the short exact sequence of Corollary 
\ref{cor:qsusp}
is obtained from the short exact sequence in $\qmgr$:
\[
 0
\rightarrow 
\big(\Phi \sing M\big)(1)
\stackrel{Q_0}{\rightarrow} 
\sing M
\stackrel{\epsilon_M}{\rightarrow} 
\Sigma^{-1} \sing \Sigma M
\rightarrow 
0.
\] 
\end{enumerate} 
\end{prop}

\begin{proof}
 Straightforward.
\end{proof}

The  following is the  length-graded version of Proposition \ref{prop:qsusp}:

\begin{prop}
\label{prop:qsuspgr}
 The exact functor 
$
 \Sigma^{-1} : \qmgr \rightarrow \qmgr 
$
admits a left adjoint $\qsusp : \qmgr \rightarrow \qmgr$ which fits into a 
natural exact sequence 
\[
 0
\rightarrow 
\Sigma^{-1} \qsusp_1 N 
\rightarrow 
\big(\Phi N\big)(1)
\stackrel{Q_0}{\longrightarrow}
N 
\rightarrow 
\Sigma^{-1} \qsusp N 
\rightarrow 
0
\]
for $N \in \qmgr$.
\end{prop}

\subsection{Length truncations}

\begin{defn}
 For $l \in \zed$, let $\tau\lgth{\leq l} : \qmgr \rightarrow \qmgr$ denote the 
length truncation functor defined by 
 \[
  (\tau\lgth{\leq l} N)\lgth{s} 
  = 
  \left\{
  \begin{array}{ll}
   N\lgth{s} & s \leq l \\
   0 & s > l,
  \end{array}
\right.
 \]
for $N \in \qmgr$. 
\end{defn}

Clearly one has the following:

\begin{prop}
\label{prop:trunc}
For $l \in \zed$, the functor $\tau \lgth{\leq l} : \qmgr \rightarrow \qmgr$ is 
exact 
and there is a commutative diagram of natural surjections:
\[
 \xymatrix{
 N 
 \ar@{->>}[r]
  \ar@{->>}[dr]
  &
  \tau\lgth{\leq l+1} N
   \ar@{->>}[d]
   \\
   &
   \tau\lgth{\leq l}N,
 }
\]
for $N\in \qmgr$.
 \end{prop}

\subsection{Connectivity estimates for $\qsusp_1$}

The following elementary result is the basis for the stable range which appears 
in many situations:

\begin{lem}
\label{lem:Phi_conn}
 For $M \in \amod$ of connectivity $d-1$, the module $\Phi M$ is 
$(2d-1)$-connected.
\end{lem}

\begin{prop}
\label{prop:qsusp_1_conn}
 For $N \in \qmgr$ such that $\tau\lgth{\leq -1} N= 0$ and, for $i \geq 0$, 
$N\lgth{i}$ is $d_i -1$ connected for $d_i \in \zed$, 
 \begin{enumerate}
  \item 
  $(\qsusp_1 N) \lgth{0}=0$; 
  \item 
  $(\qsusp_1 N) \lgth{i}$ is at least $(2d_{i-1})$-connected, for $i>0$.
 \end{enumerate}
\end{prop}

\begin{proof}
 A straightforward consequence of the exact sequence of Proposition 
\ref{prop:qsuspgr} together with Lemma \ref{lem:Phi_conn}.
\end{proof}

\section{Relative left derived functors of $\indec$}
\label{sect:relhom}

The derived functors of the Dyer-Lashof indecomposables  that are introduced in 
this section play a central rôle in the paper.

\subsection{The class of relative projectives and relative left derived 
functors}

The adjunction $\sing : \amod \rightleftarrows \qm : \mathrm{Forget}$ defines a 
projective class in $\qm$, in the sense 
of relative homological algebra. The class of projective objects is $\{\sing M 
| M \in \amod\}$.

The comonad associated to the adjunction will be denoted $\sing : \qm 
\rightarrow \qm$ (omitting the forgetful functor
from the notation), equipped with the counit $ \mu : \sing \rightarrow 1_{\qm}$ 
(corresponding to the action map) and 
 $\Delta : \sing \rightarrow \sing \sing$, induced by the adjunction unit 
$1_\amod \rightarrow \sing$. For $N \in \qm$, this 
induces a simplicial object with $n$th term 
$\sing^{n+1} N$ and equipped with the augmentation $\sing N \rightarrow N$. 

The associated  chain complex is acyclic (as is seen, as usual, 
by 
applying the forgetful functor and applying  the contracting homotopy for the 
augmented chain complex which is provided by 
the adjunction unit). Thus $\sing^{\bullet +1}N$ provides a functorial 
(relative) 
projective resolution of $N$.

\begin{defn}
 For $F: \qm \rightarrow \calc$ a right exact additive functor to an abelian 
category 
$\calc$, the (relative) left derived functors $\lder_i F : \qm \rightarrow 
\calc$ 
are defined by 
\[
 \lder_i F (N):= H_i (F (\sing^{\bullet +1} N)),
\]
so that $\lder_0 F = F$.
\end{defn}

\begin{rem}
\ 
\begin{enumerate}
 \item 
 These left derived functors can be calculated with respect to any 
relative projective resolution.
\item 
For current purposes, one could simply define these as cotriple 
derived functors.
\end{enumerate}
\end{rem}

\begin{exam}
\label{exam:qsusp}
The functor $\qsusp : \qm \rightarrow \qm$ is a left adjoint, hence is right 
exact. 
The left derived functors $\lder_i \qsusp$ are trivial for $i >1$ and $\lder_1 
\qsusp \cong \qsusp_1$, 
the functor appearing in Proposition \ref{prop:qsusp}.
\end{exam}

\begin{exam}
\label{exam:derive_indec}
The indecomposables functor $\indec : \qm \rightarrow \amod$ is right exact, 
hence there are  
 derived functors 
\[
 \lder_i \indec : \qm \rightarrow \amod.
\]
For $N \in \qm$, since $\indec \sing$ is the identity functor on $\amod$ (by 
Lemma \ref{lem:indec}), the complex $\indec \sing^{\bullet +1} N$ has the form 
\[
 \ldots \rightarrow \sing^2 N \rightarrow \sing N \rightarrow N 
\]
in $\amod$, equipped with the augmentation $N \twoheadrightarrow \indec N$. The 
morphism $\sing N \rightarrow N$ is the difference between 
the $\sing$-action structure morphism and the projection $\sing N 
\twoheadrightarrow \indec \sing N \cong N$.

In particular, this complex gives an exact functor from $\qm$ to the category 
$\ch \amod$ of  chain complexes 
in $\amod$. 
\end{exam}

\begin{defn}
Let $\qcx : \qm \rightarrow \ch \amod$ be the exact functor $\indec 
\sing^{\bullet +1}$ of Example \ref{exam:derive_indec}. 
\end{defn}

\begin{nota}
For $M\in \amodgr$,  let $\overline{\sing}M \subset \sing M$ denote the kernel 
of the natural projection 
$\sing M \twoheadrightarrow M$.
\end{nota}

The resolution $\sing^{\bullet +1} N$, for $N \in \qm$, has a reduced 
subobject, with $i$th term
\[
 \sing (\overline{\sing}) ^i N 
\]
(the analogue of the reduced bar construction), which is again a 
resolution. 

\begin{defn}
 Let $\redqcx : \qm \rightarrow \ch \amod$ denote the sub-complex of $\qcx$ 
given by applying $\indec$ to the reduced resolution, so that 
 $(\redqcx N)_i = (\overline{\sing})^i N$.  
\end{defn}

\begin{prop}
\label{prop:reduced_qcx}
For $N \in \qm$, the inclusion $\redqcx N \hookrightarrow \qcx N$ is a 
quasi-isomorphism. 
\end{prop}

\begin{proof}
 Standard.
\end{proof}

\subsection{First properties of $\lder_* \indec$}

\begin{prop}
\label{prop:proj_class_acyclic}
Objects of the projective class of $\qm$ are acyclic, namely for $N= \sing M$ 
(where $M \in \amod$), 
the augmented chain complex 
\[
 \sing^{\bullet +1} N \rightarrow N
\]
is acyclic.  In particular, $\lder _i F (\sing M) = 0$ for $i >0$ and $\lder_0 
F (\sing M) = 
F (\sing M)$.
\end{prop}

\begin{proof}
 Standard: the adjunction provides an extra degeneracy, hence a contracting 
homotopy in $\qm$. 
\end{proof}

\begin{exam}
 For $M \in \amod$, $\lder_i \indec (\sing M) =0$ for $i >0$ and $\lder_0 
\indec (\sing M) = M$.
\end{exam}

\begin{prop}
\label{prop:les}
 For $0 \rightarrow N_1 \rightarrow N_2 \rightarrow N_3 \rightarrow 0$ a short 
exact sequence in $\qm$, there 
is a natural long exact sequence of derived functors of $\indec$:
\[
 \ldots 
\rightarrow 
\lder_i \indec N_1 
\rightarrow 
\lder_i \indec N_2 
\rightarrow 
\lder_i \indec N_3 
\rightarrow 
\lder_{i-1} \indec N_1
\rightarrow 
\ldots .
\]
\end{prop}

\begin{proof}
 Applying the exact functor $\qcx : \qm \rightarrow \ch \amod$ to the short 
exact sequence gives a short exact sequence of 
chain complexes 
\[
 0 \rightarrow \qcx N_1 \rightarrow \qcx N_2 \rightarrow \qcx N_3 \rightarrow 0.
\]
The long exact sequence is given by passage to homology.
\end{proof}

The previous results can be made more precise by using the length grading, 
since 
the functor $\sing : \amod \rightarrow \qm$ extends to $\sing : \amodgr 
\rightarrow \qmgr$ 
and the forgetful functor respects length grading. 

\begin{prop}
\label{prop:length_gr_Lder_q}
\ 
\begin{enumerate}
 \item 
For $N \in \qmgr$, the augmented chain complex $\sing^{\bullet +1} N 
\rightarrow 
N$ is defined in $\qmgr$.
\item 
If $N = \sing M \in \qmgr$, for $M \in \amodgr$, then this augmented chain 
complex is acyclic.
\item 
The functor $\qcx$ extends to an exact functor $\qcx : \qmgr \rightarrow \ch 
\amodgr$.
\item
The functor $\qcx$ commutes with the length grading shift functor $\cdot (l)$ 
for $l \in \zed$; 
namely there is a natural isomorphism $\qcx \big(N (l)\big)\cong \big(\qcx 
N\big) (l)$. 
\item  
The derived 
functors $\lder_* \indec$ induce functors 
\[
 \lder_i \indec : \qmgr \rightarrow \amodgr.
\]
\item 
For $M \in \amod$, $\lder_i \indec \sing M$ is zero for $i >0$ and $\lder_0 
\indec \sing M \cong M$, considered
as concentrated in length $0$.
\end{enumerate}
\end{prop}

\begin{proof}
 Straightforward.
\end{proof}

The following observation is fundamental (and is analogous to a result proved in 
\cite[Section 3]{Miller}).

\begin{prop}
 \label{prop:lder_indec_on_trivials_exact}
 For $s \in \nat$, the composite functor 
 \[
  \amod 
  \stackrel{\triv}{\rightarrow} 
 \qmgr
\stackrel{\lder_s \indec}{  \rightarrow} 
 \amodgr
 \]
is exact.
\end{prop}

\begin{proof}
 Since the 
Dyer-Lashof action is trivial, the underlying bigraded vector space of $\lder_s 
\indec
M$ only depends upon the underlying graded vector space $M$, 
 whence the result follows by semisimplicity of the category of graded vector 
spaces and Proposition \ref{prop:length_gr_Lder_q}. (This may also be seen by 
inspection of the complex.)
\end{proof}

The functor $\lder_s \indec : \amod \rightarrow \amodgr$, when restricted to 
the 
full subcategory 
$\amod_{\geq 1}$ of $(-2)$-connected objects, is  identified 
explicitly in Theorem \ref{thm:Koszul_property}, which leads to a Koszul 
complex 
calculating these derived functors (see Corollary 
\ref{cor:vanishing_lder_indec}).

\subsection{Relating $\lder_* \indec$ to Miller's $\untor_*(\field, -)$}

Miller  \cite[Section 2.2]{Miller} introduced the category $\anr[n]$ of 
non-negatively graded $n$-allowable modules over the Dyer-Lashof algebra. 
This can be enriched (as in \cite[Section 4]{Miller}, but not requiring 
instability of the $\cala$-module structure) to take into account the Steenrod 
action  
(which is compatible with the Dyer-Lashof action via the Nishida relations). 
This gives the category $\anar[n]$  of non-negatively graded $n$-allowable 
modules over the Dyer-Lashof and Steenrod algebras. 

\begin{rem}
 The restriction to non-negatively graded objects is implicit in \cite{Miller}.
\end{rem}

The category $\anar[0]$  is $\qm^{\geq 0}$, the full subcategory of $\qm$ of 
non-negatively 
graded objects. The category $\anar$ is the full 
subcategory of $\anar[0]$ of modules $N$ such that $\Sigma N \in \qm$ 
(equivalently, $Q_0$ acts trivially on $N$).

Forgetting the Dyer-Lashof instability condition, an object of $\anr$ is an 
$R(-\infty)$-module, where $R(-\infty)$ is the quotient of the big Dyer-Lashof 
algebra by the two-sided ideal generated by the  operations $Q^i$ with $i <0$. 
This inclusion fits into 
an adjunction 
\[
 R(-\infty)\dash\mathrm{mod}
\rightleftarrows 
\anr[1]
\]
which Miller uses to define a projective class. The analogous construction 
works for $\anar$: the projectives are 
of the form $\Sigma^{-1} \sing \Sigma M$, where $M \in \amod$ is non-negatively 
graded. 

\begin{rem}
The category $\qm$ is equivalent to the category of $\sing$-modules in $\amod$. 
The above shows that $\anar$ is the category of $\Sigma^{-1} \sing 
\Sigma$-modules in $\amod_{\geq 0}$, 
where $\amod_{\geq 0} \subset \amod$ is the full subcategory of $(-1)$-connected 
objects. 
Relative projective resolutions are associated as before to the functor 
$\Sigma^{-1} \sing \Sigma$. 

This can be seen explicitly by considering the reduced bar resolution used 
in \cite[Section 3]{Miller}, where the functor $U : \field_2\dash 
\mathrm{mod} \rightarrow \anr$ extends to 
$\amod^{\geq 0} \rightarrow \anar$ and the latter  is the restriction  of 
$\Sigma^{-1} \sing \Sigma$ to 
$\amod^{\geq 0}$.
\end{rem}

\begin{defn}
(Cf. \cite[Section 2.2]{Miller} and \cite[Section 4]{Miller}.)
 Let $\untor_* (\field, -)$ be the left derived functors of the composite 
$\anar[1] \hookrightarrow \qm \stackrel{\indec}{\rightarrow} \amod$
 with respect to the projective class defined by $\Sigma^{-1} \sing \Sigma$.
\end{defn}

\begin{prop}
\label{prop:untor_Lq}
 For $N \in \anar[1]$, there is a natural isomorphism 
\[
 \Sigma \untor_* (\field , N) \cong \lder_* \indec (\Sigma N).
\]
\end{prop}

\begin{proof}
 Straightforward (using the fact that $\indec$ commutes with $\Sigma^{-1}$).
\end{proof}

\begin{exam}
 Consider $N:= \Sigma^{-1} \sing \Sigma M$, for $M \in \amod ^{\geq 0}$. 
Observe 
that $N$ is the quotient
of $\sing M$ by the image of $Q_0$, by Proposition \ref{prop:qsusp}.

Then 
\[
 \untor_i (\field, \Sigma^{-1} \sing \Sigma M) 
\cong 
\left\{
\begin{array}{ll}
0 & * >0
\\
M & *=0.
\end{array}
\right.
\]
\end{exam}

\section{The Steinberg functors and the Koszul complex for $\lder_* \indec $}

The Steinberg functors $\stein_s$ that are introduced in this section are of independent interest. 
Their importance here is through the rôle that they play in the Koszul complex (see Definition \ref{def:kosz_cx_stein}) with homology 
 the calculating the derived functors $\lder_* \indec$ (see Corollary \ref{cor:conn_lderq_N}).

\subsection{The Steinberg functor}

The functor $\sing : \amod \rightarrow \qmgr$ has length components $\sing_s : 
\amod \rightarrow \amod$ for $s \in \nat$, where $\sing_s$ is the 
$s$th Singer functor (this should be taken as the definition of the Singer functors here).
 The adjunction unit induces natural transformations 
$$
\sing_s \sing_t \twoheadrightarrow \sing_{s+t}
$$
 and, in particular, there is a natural surjection
 \[
  (\sing_1)^s \twoheadrightarrow 
  \sing_s.
 \]
The Singer functors have a quadratic (co)presentation, and are thus determined 
by the surjection $ (\sing_1)^2 \twoheadrightarrow 
  \sing_2$. The {\em Steinberg functors} are defined by the quadratic dual 
construction.
  
  \begin{defn}
   \label{def:Steinberg}
   For $s \in \nat$, define the Steinberg functor $\stein_s : \amod \rightarrow 
\amod$ on $M \in \amod$ by:
   \[
    \stein_s M := \bigcap _{i=0}^{s-2} \ker \big \{ (\sing_1)^s M
    \rightarrow 
    (\sing_1)^i \sing_2 (\sing_1)^{s-(i+2)} M
    \big \}
   \]
for the natural transformations induced by $ (\sing_1)^2 \twoheadrightarrow 
  \sing_2$.
  \end{defn}

The following records  properties of the Steinberg functors. 
  
\begin{prop}
 \label{prop:Steinberg}
For $s \in \nat$, the following properties hold.
\begin{enumerate}
 \item 
 The functor $\stein_s$ is exact and there is a natural inclusion $\stein_s 
\hookrightarrow (\sing_1)^s$ that is an isomorphism for $s\in \{0, 1 \}$.
 \item 
 For $s=2$, there is a natural short exact sequence
 \[
  0
  \rightarrow 
  \stein_2 
  \rightarrow 
  (\sing_1)^2 
  \rightarrow 
  \sing_2
  \rightarrow 
  0.
 \]
\item 
For $s_1, s_2 \in \nat$, there is a natural inclusion
\[
 \stein_{s_1+ s_2} 
 \hookrightarrow 
 \stein_{s_1} \stein_{s_2}
\]
and, for varying $s_1, s_2$, these are coassociative. In particular, for $s>0$, 
there is a natural inclusion $\stein_s \hookrightarrow \stein _{s-1} \stein_1$. 
\end{enumerate}
\end{prop}

 \begin{proof}
 The exactness of $\stein_s$ can be checked as for $\sing_s$. Indeed, as for 
the proof of Proposition \ref{prop:lder_indec_on_trivials_exact}, the 
underlying graded vector space of $\stein_s M$ depends only 
 upon that of $M$. The remaining statements are clear.
 \end{proof}

 \begin{cor}
 \label{cor:stein_alternative}
     For $2 \leq s \in \nat$ and $M \in \amod$, there is a natural isomorphism: 
   
     \[
    \stein_s M := \bigcap _{i=0}^{s-2} 
    (\stein_1)^i \stein_2 (\stein_1)^{s-(i+2)} M,
   \]
   where the intersection is formed within $(\stein_1)^s M$.
 \end{cor}

 \begin{proof}
  Follows from the exactness of $\stein_1$ and the short exact sequence 
identifying $\stein_2$. 
 \end{proof}

This constructs the Steinberg functors $\stein_s$ as the quadratic duals 
to the Singer functors $\sing_s$. An alternative approach, again explaining the 
dual nature of the construction, uses the Hecke algebra, as in \cite[Section 
4]{Kuhn:Whitehead}, as sketched below.
  
  \begin{nota}
   For $s \in \nat$, let 
\begin{enumerate}
 \item 
$\hecke_s$ denote the Hecke algebra of type $A_{s-1}$ defined by 
   \[
    \hecke_s := \mathrm{End}_{GL_s} (\field [B_s \backslash GL_s]) 
   \]
where $B_s < GL_s$ is the Borel subgroup of upper triangular matrices (for 
$s=0$, take $\hecke_0= \field$);
\item 
$D(s)$ denote the Dickson algebra $H^* (BV_s) ^{GL_s}$ (a {\em cohomological} unstable algebra), where $V_s$ is an elementary abelian $2$-group of rank $s$; 
\item 
$D(s)\dash \unst$, the category of $D(s)$-modules in {\em 
cohomological} unstable modules.
\end{enumerate}
\end{nota}

  \begin{prop}
  \label{prop:hecke_action}
   For $s \in \nat$, the Hecke algebra $\hecke_s$ acts by natural 
transformations upon the functor $(\sing_1)^s : \amod \rightarrow \amod$.
  \end{prop}
  
  \begin{proof}
 It is more transparent to present this proof for the {\em cohomological} 
Singer functors. For unstable modules 
 these are considered by Lannes and Zarati \cite{LZ} and they extend to all 
$\cala$-modules as in \cite{p_destab} (which is written for the odd primary 
case).
 To avoid confusion, denote the cohomological Singer functors by $R_s$. 
 
For $s\in \nat$, $(R_1)^s \field$ is isomorphic to  the unstable algebra $H^* 
(BV_s)^{B_s}$ and 
$R_s \field$ is the 
Dickson algebra $D(s)$. The Hecke algebra $\hecke_s$ acts 
via morphisms of $D(s)\dash \unst$:
\[
 \hecke_s \rightarrow \mathrm{End}_{D(s)\dash \unst} (H^* (BV_s)^{B_s}). 
\]

Now, for a (cohomological) $\cala$-module $M$, $R_s M$ is a $D(s)$-module in 
the category of $\cala$-modules and 
there is a natural isomorphism of $\cala$-modules:
\[
 (R_1)^s M 
 \cong 
 H^* (BV_s)^{B_s}
 \otimes_{D(s)} 
 R_s M.
\]
The action of $\hecke_s$ on $H^* (BV_s)^{B_s}$ therefore induces a natural 
action on $(R_1)^s M$. (Moreover, this action is $D(s)$-linear.)
  \end{proof}

  The Hecke algebra $\hecke_s$ is equipped with the involution $\hat{}  : 
\hecke_s \rightarrow \hecke_s$ (see  \cite[Proposition 4.4]{Kuhn:Whitehead}) 
and contains the idempotent $e_s \in \hecke_s$ of \cite[Definition/ Proposition 
4.6]{Kuhn:Whitehead}) and the corresponding idempotent $\hat{e}_s$.

  We record the following generalization of the results of 
\cite{Kuhn:Whitehead}:

 \begin{prop}
 \label{prop:sing_stein}
  For $s \in \nat$ and $M \in \amod$, there are natural isomorphisms:
\begin{eqnarray*}
 \sing_s M & \cong & \hat{e}_s (\sing_1)^s M  \\
  \stein_s M & \cong & e_s (\sing_1)^s M.
\end{eqnarray*}
 \end{prop}

\begin{proof}
 The result follows from \cite[Corollary 4.9]{Kuhn:Whitehead} and 
\cite[Corollary 4.10]{Kuhn:Whitehead}.
\end{proof}
  
  \begin{nota}
  \cite[Proposition 2.2]{HNS}
  For $s \in \nat$, denote by 
  \begin{enumerate}
   \item 
  $M_s$ the cohomological unstable module  defined using 
the Steinberg idempotent 
  $\mathrm{St}_s \in \field [GL_s]$:
  \[
   M_s := \mathrm{St}_s H^* (BV_s)
  \]
  for $V_s$ an elementary abelian $2$-group of rank $s$;
\item 
$L_s$ the cohomological unstable module defined by the canonical decomposition
\[
 M_s \cong L_s \oplus L_{s-1}.
\]
  \end{enumerate}
  \end{nota}

 \begin{rem}
  The notation $L_s$ must not be confused with that used in 
\cite{Kuhn:Whitehead}, where a suspension is introduced. 
 \end{rem}

\begin{nota}
 Write the top Dickson invariant as $\omega_s \in 
D(s)$ (this is the product of the non-zero 
classes of $H^1(BV_s)$ so that  $|\omega_s |= 2^s -1$; see \cite{HNS}, for 
example).
\end{nota}

 \begin{rem}
 \label{rem:cohom_stein}
 By construction, $M_s$ belongs to $D(s)\dash \unst$ and,  
 by \cite[Proposition 2.3]{HNS}, $L_s = \omega_s M_s$, in particular  $L_s 
\subset M_s$ is a sub-object in $D(s)\dash\unst$. 

In the cohomological setting, it follows from \cite[Section 
4.2]{Kuhn:Whitehead} that the Steinberg functor corresponds to the functor
\[
M\mapsto  M_s \otimes_{D(s)} R_s M.
\]
 \end{rem}

\begin{cor}
\label{cor:stein_field_Ls}
 For $s \in \nat$, there is a natural isomorphism in $\amod$:
 \[
  \stein_s \field 
  \cong 
  M_s^*.
 \]
\end{cor}

\begin{proof}
 This follows from Proposition \ref{prop:sing_stein}.
\end{proof}

\begin{rem}
Using  \cite[Theorem 6.3]{Kuhn:Whitehead}, Kuhn deduces that $\stein_s \Sigma 
\field \cong \Sigma L_s^*$ (using the notation adopted here, rather than 
Kuhn's), so that the right hand side should be understood as $\Sigma (\omega_s M_s)^*$.
\end{rem}

 \subsection{The Steinberg functor and suspension} 
  
  \begin{prop}
  \label{prop:stein_sigma_surject}
 The natural surjection $\Sigma \sing_1 \twoheadrightarrow \sing_1 \Sigma$, 
given by the restriction of $\epsilon : \Sigma \sing \twoheadrightarrow \sing 
\Sigma$ to length one, induces 
 a natural surjection for $s\in \nat$
 \[
  \Sigma \stein_s \twoheadrightarrow \stein_s \Sigma.
 \]
  \end{prop}

 \begin{proof}
  Clearly $\epsilon$ induces a natural surjection $\Sigma(\sing_1 )^s 
\twoheadrightarrow (\sing_1)^s \Sigma$. The result follows by applying the 
idempotent $e_s$.
 \end{proof}

The above suspension morphism fits into a short exact sequence analogous to 
that of Corollary \ref{cor:qsusp}.

 \begin{prop}
 \label{prop:ses_stein_susp}
  For $M \in \amodc$ that is of finite type and $s \in \nat$, the suspension 
morphism fits into a natural short exact sequence 
  \[
   0
   \rightarrow 
   \Sigma^{-1} \stein_{s-1} \Sigma \Phi M 
   \rightarrow 
   \stein_s M 
   \rightarrow 
   \Sigma^{-1} \stein_s \Sigma M 
   \rightarrow 
   0.
  \]
 \end{prop}

 \begin{proof}
  It is convenient to give the proof working in cohomological $\cala$-modules, 
using the hypothesis that $M$ is of finite type and bounded below to translate 
to this setting. 
  
  Using the cohomological Singer functors $R_s$, there is a natural short exact 
sequence 
  \[
   0
   \rightarrow 
   \Sigma^{-1} R_s \Sigma N
   \rightarrow 
   R_s N
   \rightarrow 
   \Phi R_{s-1} N
   \rightarrow 
   0
  \]
of $D(s)$-modules in $\amod$, where $D(s)$ acts upon $\Phi R_{s-1}N$ via the 
natural surjection $D(s) \twoheadrightarrow \Phi D(s-1)$ which has kernel 
$\omega_s D(s)$.

The underlying $D(s)$-module of $M_s$ is free (cf. for example the $2$-primary case of \cite[Corollary 3.11]{Mitchell}), hence applying 
 the functor $(M_s \otimes_{D(s)} - )$ yields an exact sequence 
\[
   0
   \rightarrow 
   \Sigma^{-1} L'_s \Sigma N
   \rightarrow 
   L'_s N
   \rightarrow 
   M_s \otimes_{D(s)} \Phi R_{s-1} N
   \rightarrow 
   0
  \]
  where $L'_s$ denotes the functor $M_s \otimes_{D(s)} R_s (-) $, which corresponds to $\stein_s$, by Remark \ref{rem:cohom_stein}. 
  
  The inclusion $\Phi D(s-1) \hookrightarrow D(s-1)$ of unstable algebras 
induces a canonical isomorphism of $D(s-1)$-modules (in $\cala$-modules):
  \[
   D(s-1) \otimes_{\Phi D(s-1) } \Phi R_{s-1} N 
   \cong 
   R_{s-1} \Phi N. 
  \]
Hence, there is a natural isomorphism
\[
 M_s \otimes_{D(s)} \Phi R_{s-1} N
 \cong 
 L_{s-1} \otimes _{D(s-1)} R_{s-1}\Phi N,
\]
using the identification $M_s / \omega_s M_s \cong L_{s-1}$ as $D(s)$-modules, 
where $D(s)$ acts on $L_{s-1}$ via the composite $D(s) \twoheadrightarrow \Phi 
D(s-1) \hookrightarrow D(s-1)$ and the $D(s-1)$-module structure of $L_{s-1}$. 

Now, $L_{s-1}$ identifies as $\omega_{s-1} M_{s-1}$, hence there is a natural 
isomorphism 
\[
 L_{s-1} \otimes _{D(s-1)} R_{s-1}\Phi N
 \cong 
 \Sigma^{-1} M_{s-1} \otimes_{D(s-1)} R_{s-1}\Sigma \Phi N
\]
and, by definition, the latter is $\Sigma^{-1}L'_{s-1} \Sigma \Phi N$. This 
provides the required short exact sequence.
 \end{proof}

\begin{rem}
 For $s=1$, the short exact sequence is 
 \[
   0
   \rightarrow 
  \Phi M 
   \rightarrow 
   \stein_1 M 
   \rightarrow 
   \Sigma^{-1} \stein_1 \Sigma M 
   \rightarrow 
   0
  \]
(where $\stein_1$ identifies with $\sing_1$).

For $s>1$, by composing with the natural surjection $\stein_{s-1} \Phi M 
\twoheadrightarrow \Sigma^{-1} \stein_{s-1} \Sigma \Phi M$, 
the inclusion of the kernel induces a natural transformation
\begin{eqnarray}
 \label{eqn:stein_Phi}
 \stein_{s-1} \Phi M 
 \rightarrow 
 \stein_s M.
\end{eqnarray}
Upon composing with the natural inclusion $\stein_s M \hookrightarrow 
\stein_{s-1} \stein_1 M$, one obtains a natural transformation 
$\stein_{s-1}\Phi M \rightarrow  \stein_{s-1} \stein_1 M$. This is {\em not} 
the 
natural inclusion induced by $\Phi M \hookrightarrow \stein_1 M$. 

Rather, the natural transformation (\ref{eqn:stein_Phi}) is given as the 
composite
\[
 \stein_{s-1} \Phi M \rightarrow \stein_{s-1} \stein_1 M \twoheadrightarrow 
\stein_s M
\]
of this inclusion with the projection induced by the Steinberg 
idempotent. The fact that this factorizes across $\Sigma^{-1} \stein_{s-1} 
\Sigma \Phi M$ follows, since the Steinberg relation imposes `complete unallowability' (as in \cite{Kuhn:Whitehead}).
\end{rem}

 \subsection{The Steinberg functor and instability}

 It is important to understand the behaviour of the functors $\sing_s$, 
$(\sing_1)^s$ and $\stein_s$ when applied to 
 modules which are iterated desuspensions of unstable modules. It is a 
fundamental fact that these functors 
 preserve instability.

 \begin{prop}
 \label{prop:iterated_sing_instability}
  For $0 < s \in \nat$ and $M \in \unst$, 
  \begin{enumerate}
   \item 
   $(\sing_1)^s M$ is unstable; 
   \item 
 for    $0< d \in \nat$,
  $(\sing_1)^s \Sigma^{-d}M $ admits a finite filtration with associated graded 
$\gr \big((\sing_1)^s \Sigma^{-d}M \big) $ such that 
  \[
   \Sigma^{2^s d} \gr \big((\sing_1)^s \Sigma^{-d}M \big)
  \]
is unstable.
  \end{enumerate}
 \end{prop}

 \begin{proof}
  Since $\sing_1$ is an exact functor, by induction upon $s$ it suffices to 
consider the case $s=1$. This case is treated by induction upon $d$, the case 
$d=0$ following 
  from the fact that $\sing_1$ restricts to a functor $\sing_1 : \unst 
\rightarrow \unst$. 
  
  For the inductive step, consider the exact sequence derived from Proposition 
\ref{prop:qmgr}:
  \[
   0
   \rightarrow 
   \Phi \Sigma^{-d} M
 \rightarrow
 \sing_1 \Sigma^{-d}M
 \rightarrow
 \Sigma^{-1}\sing_1 \Sigma ^{-(d-1)} M
 \rightarrow 
 0.
 \]
The right hand term is treated by the inductive hypothesis and the left hand 
term by using the natural isomorphism $  \Phi \Sigma^{-d} M \cong \Sigma^{-2d} 
\Phi M$, where $\Phi M$ is unstable.
 \end{proof}

 \begin{cor}
 \label{cor:sing_stein_instability}
 For $0 < s \in \nat$ and $M \in \unst$, 
  \begin{enumerate}
   \item 
   $\sing_s M$ and $\stein_s M$ are unstable; 
   \item 
 for    $0< d \in \nat$,
  $\sing_s\Sigma^{-d}M $ (respectively $\stein_s \Sigma^{-d}M$)  admit  finite 
filtrations with associated gradeds $\gr \big(\sing_s \Sigma^{-d}M \big) $ 
(respectively $\gr \big(\stein_s \Sigma^{-d}M \big) $) such that 
  \begin{eqnarray*}
   \Sigma^{2^s d} \gr \big(\sing_s\Sigma^{-d}M \big)\\
    \Sigma^{2^s d} \gr \big(\stein_s\Sigma^{-d}M \big)
  \end{eqnarray*}
are unstable.
  \end{enumerate}  
 \end{cor}

 \begin{proof}
  A straightforward consequence of Proposition 
\ref{prop:iterated_sing_instability}.
 \end{proof}

 \begin{rem}
  The analysis of $\stein_s \Sigma^{-d}$ can be refined by using induction upon 
$s$ and $d$ together with the short exact sequence provided by Proposition 
\ref{prop:ses_stein_susp}.
 \end{rem}

 \subsection{Identifying $\stein_s \Sigma^d \field$}

 Corollary \ref{cor:stein_field_Ls} identifies $\stein_s \field$, for $s \in 
\nat$; this can be generalized to   consider $\stein_s \Sigma ^d \field$ for 
all $d \in \zed$. It is conceptually clearer to present this  in the 
cohomological setting.

The localized algebra $D(s)[\omega_s^{-1}]$ is 
an algebra in cohomological $\cala$-modules and the construction of the 
Singer functors shows that 
\[
  \Sigma^{-d} R_s \Sigma^d \field  = \omega_s ^d D(s) \subset D(s)[\omega_s^{-1}],
\]
for any $d\in \zed$, as $D(s)$-modules in $\cala$-modules. 
 
As in Remark \ref{rem:cohom_stein}, it follows that there is an isomorphism 
\[
 (\Sigma^{-d} \stein_s \Sigma^d \field) ^* \cong \omega_s ^d M_s 
\]
where the right hand side is understood as a sub-object of $M_s [\omega_s^{-1}] 
:= M_s \otimes_{D(s)} D(s)[\omega_s^{-1}]$. (Here, $\omega_s ^d M_s$
 is of finite type, so the duality causes no difficulty.) 
 
 In particular, for $d \geq 0$, this gives a decreasing filtration with 
 \[
  \omega_s^d M_s =  \omega_s^{d-1} L_s\subseteq L_s \mathrm{,\  for\  }d>0.
 \]

For current purposes, it is sufficient to describe bases of the underlying 
graded vector spaces.  Recall the following standard definition:

  \begin{defn}
\label{def:sequence_admissibility}
\ 
\begin{enumerate}
 \item 
 A sequence of length $s$ (for $s \in \nat$) is an ordered sequence  $\{i_1, 
\ldots , i_s \} \in \zed^s$ (empty if $s=0$).
 The degree of $I$ is $d (I):= \sum_j i_j$ and the length $s$ is denoted by 
$l(I)$.
 \item 
 The sequence $I$ is admissible if $i_j \geq 2 i_{j+1}$ for all $0 \leq j< s$; 
the excess of an admissible sequence $I$ is 
 $e(I) = \sum_j (i_j - 2 i_{j+1})$. 
 \end{enumerate}
\end{defn}
  
 \begin{nota}
  For applications here, sequences will usually lie in $\nat^s$ and, often, 
will be sequences of positive integers (that is lying in $\nat_*^s$). 
  
However, much of this material should be considered in the context of the big 
Steenrod algebra (cf. \cite{BCL} and \cite{p_quad}), where the indexing is by 
$\zed$.
 \end{nota}
 
  \begin{prop}
  \label{prop:admissible_basis}
For $0<s \in \nat$ and $d\in \zed$, the graded vector space $\stein_s \Sigma^d 
\field$ has a basis indexed by admissible sequences of length $s$:
   \[
   \{
\sigma_I \iota_{d-s} |    \mathrm{admissible, \ } i_s > d 
   \}
   \]
where $\sigma_I \iota_{d-s}$ has degree $d(I)+ d -s$. 
\end{prop}

\begin{proof} (Indications.)
This is best understood in the cohomological setting, so that the above refers 
to a dual basis. For $d=0$, this is standard. The effect of multiplying by $\omega_s$ is to shift an 
admissible sequence by the excess zero sequence $(2^{s-1}, 2^{s-2}, \ldots , 
1)$ of length $s$.
\end{proof}

 \begin{rem}
For $d\geq 1$, this result can be deduced from  from \cite[Proposition 
3.3.11]{Miller} by  using  Proposition 
\ref{prop:untor_Lq}. 
\end{rem}

For completeness, the underlying relationship between Dyer-Lashof operations 
and admissible sequences is explained below. Proposition 
\ref{prop:admissible_basis} can also be deduced
from Lemmas \ref{lem:basis_stein}, \ref{lem:lower_2_upper} and 
\ref{lem:reindex_admissible}.

  \begin{nota}
   For $0<s \in \nat$ and  $J \in \nat^s$ an ordered sequence of non-negative 
integers, write $Q_J$ for the iterated operation $Q_J := Q_{j_1} Q_{j_2} \ldots 
Q_{j_s}$, to be interpreted either as a formal composition (as in 
$(\stein_1)^s=(\sing_1)^s$) or 
as the quotient in $\sing_s$ (respectively $\stein_s$). 
  \end{nota}

  \begin{lem}
  \label{lem:basis_sing}
   For $0< s \in \nat$ and $d \in \zed$, 
   \begin{enumerate}
    \item 
    $(\sing_1)^s \Sigma^d \field$ has a basis $\{ Q_J \iota_d | J \in \nat^s 
\}$; 
    \item 
    $\sing_s \Sigma^d \field$ has a basis $\{ Q_J \iota_d | J \in \nat^s, \  
j_i\leq j_{i+1} \mathrm{\ for \ } 1 \leq i < s  \}$, 
   \end{enumerate}
   where $\iota_d$ denotes the generator of $\Sigma^d \field$.
  \end{lem}

  \begin{proof}
   This follows from the proof of \cite[Lemma 4.19]{KMcC}.
  \end{proof}

  \begin{lem}
  \label{lem:basis_stein}
   For $0<s \in \nat$ and $d \in \zed$, $\stein_s \Sigma^d \field$ has a basis 
   \[
    \{ Q_J \iota_d | J \in \nat^s, \  j_i > j_{i+1} \mathrm{\ for \ } 1 \leq i 
< s  \}
   \]
  \end{lem}

  \begin{proof} (Sketch.)
   The case $s=2$ follows  from Lemma \ref{lem:basis_sing} and then 
the general case from the construction of $\stein_s$, which corresponds to the 
assertion that there is a PBW-basis of this form. This is equivalent to the 
Koszul property
 {\em à la Priddy} \cite{Priddy}; see Section \ref{subsect:Koszul} below.
\end{proof}

  \begin{rem}
   The previous result should be compared with Kuhn's comments following 
\cite[Theorem 6.3]{Kuhn:Whitehead}, where $\hat{e}_s$ is described as rewriting 
terms in Dyer-Lashof allowable form and 
   $e_s$ as rewriting in Dyer-Lashof `completely unallowable' form.
  \end{rem}

  To proceed, Dyer-Lashof operations must be rewritten with  upper indexing:
  
  \begin{lem}
   \label{lem:lower_2_upper}
   For $0 < s \in \nat$ and $d \in \zed$:
   \[
    Q_J \iota_d 
    = 
    Q^{\alpha_1} \ldots Q^{\alpha_s} \iota_d,
   \]
where $\underline{\alpha}= \underline{\alpha(J, d)}\in \zed^s$ is the ordered 
sequence of integers given by 
\[
\alpha_t := 
i_t + \sum_{l >t} 2^{l- (t+1) } j_l + 2^{s-t} d.
 \]
 \end{lem}

 \begin{proof}
  Induction upon $s$.
 \end{proof}

 \begin{nota}
  For $\underline{\alpha} \in \zed^s$, write $\underline{\alpha +1}$ for the 
sequence $(\alpha_1 +1, \ldots , \alpha_s +1)$.
 \end{nota}

\begin{lem}
 \label{lem:reindex_admissible}
 For $d \in \zed$, the association $J \mapsto \underline{\alpha(J,d) +1}$ 
induces a bijection between $\nat^s$ and 
 \[
  \{
  I \in \zed^s | \mathrm{admissible, \ } i_s> d 
  \}.
 \]
 If $d\geq 0$, then $\underline{\alpha(J,d) +1} \in \nat_* ^s$;  
$\underline{\alpha(J,-1) +1}\in \nat^s$.
\end{lem}

 \begin{proof}
  It is straightforward to check that $J  \mapsto \underline{\alpha(J,d) +1}$ 
is one to one, hence it suffices to identify the image. 
  Since $\alpha_s = j_s +d$ and $j_s \geq 0$, the condition on $i_s$ follows 
immediately. 
  Now, for $0 \leq t < s$, $\alpha_t - 2 \alpha_{t+1} = j_t - j_{t+1}$; it 
follows that $\underline{\alpha (J,d) +1}$ is admissible if and only if $j_t > 
j_{t+1}$ for all $t$.
  This establishes the bijection; the final statement is clear.
 \end{proof}

 \begin{rem}
  The above re-indexing is intimately related to the quadratic duality of 
\cite{p_quad}.
 \end{rem}

\subsection{Connectivity}

Proposition \ref{prop:admissible_basis} leads to an understanding of the 
connectivity of objects 
of the form $\stein_t M$, for $M \in \amodc$, by applying the following result:

\begin{prop}
 \label{prop:conn_stein}
 For $t \in \nat$ and $d \in \zed$, $\stein _t \Sigma^d \field$ is precisely 
  $\big(2^t (d+1) -(t+2)\big)$-connected (that is, the lowest degree class is 
in 
degree $2^t (d+1) -(t+1)$).

  In particular, the lowest degree class of $\stein_t \Sigma^{-1} \field$ is in 
degree $-(t+1)$.
  \end{prop}

\begin{proof}
 This result can be proved either by using Proposition 
\ref{prop:admissible_basis} or directly from Lemma \ref{lem:basis_stein}. 
 (The reader is encouraged to use the Lemma  in the case $d=-1$ and 
then deduce the general case.)
\end{proof}

\begin{cor}
\label{cor:conn_stein}
For $M \in \amod$ which is $(d-1)$-connected and $t \in \nat$, $\stein_t M$ is 
$\big(2^t (d+1) - (t+2)\big)$-connected
hence is at least $(2^t d -1)$-connected.  
\end{cor}
 
 \begin{proof}
  The first statement is an immediate consequence of Proposition 
\ref{prop:conn_stein}. 
  The second follows from 
  \[
   2^t - (t+2) \geq -1,
  \]
  for $t \in \nat$, where the inequality is strict for $t \geq 2$.
 \end{proof}

\begin{rem}
\label{rem:improve_conn}
Clearly the first statement gives a much better bound for connectivity for 
large 
$t$. The weaker bound is sometimes more convenient for 
describing generic behaviour. 
\end{rem}

\subsection{The Koszul property}
\label{subsect:Koszul}

In \cite[Section 3]{Miller}, Miller observes that Priddy's results on Koszul 
duality \cite{Priddy} carry over to $\untor_* (\field, -)$. This is also true 
for $\lder_* \indec$.

\begin{thm}
 \label{thm:Koszul_property}
 For $M \in \amod_{\geq -1}$, considered as an object of $\qmgr$ via $\triv$, 
 $\big (\lder_s \indec M \big) \lgth{i}=0$ if $i \neq s$ and there is a natural 
isomorphism in $\amod$:
 \[
 \big (\lder_s \indec M \big) \lgth{s}
 \cong 
 \stein_s M.
 \]
\end{thm}

\begin{proof}
For the first statement, by exactness of $\lder_s \indec$ restricted to $\amod$ 
(see Proposition \ref{prop:lder_indec_on_trivials_exact}), it suffices to 
consider the case $M = \Sigma^d \field$. 
Here the result follows as for \cite[Proposition 3.1.2]{Miller}; indeed, for $d 
\geq 1$, the statement can be deduced from this result by  using  Proposition 
\ref{prop:untor_Lq}. (The cases $d\in \{ -1, 0 \}$ can then be deduced from 
this by using the spectral sequence of Section \ref{sect:ss}.)

The identification of $ \big (\lder_s \indec M \big) \lgth{s}$ follows by 
considering the reduced complex $\redqcx M$ in length degree $s$ (see 
Proposition \ref{prop:reduced_qcx}). In  homological degree $s$ this identifies 
the cycles as the kernel of the map defining $\stein_s M$ and there are no 
non-trivial boundaries.
\end{proof}

\begin{rem}
The hypothesis $d\geq -1$ for the vanishing of $(\lder_s \indec M )\lgth{i}$  
(for $i \neq s$) can be relaxed by using the Koszul duality result of 
\cite{BCL} 
(cf. also \cite{p_quad}).
\end{rem}

If $N \in \qmgr$, then the Dyer-Lashof action induces natural transformations: 
\[
 \sing_1 N\lgth{i}
 \rightarrow 
 N \lgth{i+1}
\]
for $i \in \zed$. Since $\stein_1 = \sing_1$, composing with the natural 
inclusion $\stein_s \rightarrow \stein_{s-1}\stein_1$ (for $s=0$ this is taken 
to be zero) leads to the natural `Koszul differential':
\[
 d^\stein_s :\stein_s N\lgth{i}
 \rightarrow 
 \stein_{s-1} N \lgth{i+1}
\]
which can be considered as a natural transformation in $\amodgr$:
\[
 \stein_ * N 
 \rightarrow 
 \stein_{*-1}\big( N (-1) \big).
\]

\begin{lem}
 \label{lem:d^2=0}
 For $N \in \qmgr$ and $s\in \nat$, the composite 
 \[
  \stein_s N\lgth{i}
 \stackrel{ d^\stein_s }{\rightarrow} 
 \stein_{s-1} N \lgth{i+1}
 \stackrel{ d^\stein_{s-1} }{\rightarrow} 
 \stein_{s-2} N \lgth{i+2}
 \]
is trivial.
\end{lem}

\begin{proof}
 This is the standard Koszul complex argument: since the `coproducts' 
$\stein_{s+t} \rightarrow \stein_s \stein_t$ are coassociative (by Proposition 
\ref{prop:Steinberg}), that $d^2=0$ follows 
from the fact that $\stein_2$ is the kernel of $(\sing_1)^2 \twoheadrightarrow 
\sing_2$. 
\end{proof}

\begin{defn}
\label{def:kosz_cx_stein}
 For $N \in \qmgr$ and $n \in \nat$, let $\kz[n] N$ denote the complex in 
$\amod$:
 \[
  \stein_n N\lgth{0}
  \stackrel{d^\stein_n }{\rightarrow}
  \stein_{n-1} N\lgth{1}
  \stackrel{d^\stein_{n-1} }{\rightarrow}
\stein_{n-2} N\lgth{2}
\rightarrow 
\ldots 
\rightarrow 
\stein_0 N \lgth{n} = N\lgth{n},
  \]
where $\stein_{i} N\lgth{n-i}$ is placed in homological degree $i$.
\end{defn}

\begin{cor}
 \label{cor:vanishing_lder_indec}
(Cf. \cite[Theorem 3.3.16]{Miller}.)
 Let $N \in \qmgr$ such that $\tau\lgth{\leq -1} N =0$ and $N\lgth{i} \in 
\amod_{\geq -1}$ for each $i$.  For $ s, t \in \nat$, 
 there is a natural isomorphism
 \[
  \big( \lder_t \indec N) \lgth{s} 
  \cong 
  H_t (\kz N).
 \]
In particular,
 $
 \big( \lder_t \indec N \big) \lgth{s} =0
 $
for integers $t > s\geq 0$ and 
\[
  \big( \lder_s \indec N) \lgth{s} 
  \cong 
  \ker \big\{
  \stein_s N\lgth{0} 
  \stackrel{d^\stein_s}{\rightarrow}
  \stein_{s-1} N \lgth{1}
 \big \}.
\]

\end{cor}

\begin{proof}
The spectral sequence associated to the length filtration of $N$ degenerates to 
the Koszul complexes, by Theorem \ref{thm:Koszul_property} (which uses the 
hypothesis $N\lgth{i} \in \amod_{\geq -1}$).
\end{proof}

\begin{exam}
\label{exam:lder_q_one}
 For $N \in \qmgr$ such that $\tau\lgth{\leq -1}N=0$ and $N\lgth{i} \in 
\amod_{\geq -1}$ for each $i$, there are natural 
isomorphisms:
 \begin{eqnarray*}
  (\lder_1 \indec N)\lgth{s} &\cong& \ker\{ \sing_1 N\lgth{s-1} \rightarrow 
N\lgth{s} \}
\\
(\lder_0 \indec N)\lgth{s} &\cong& \mathrm{coker} \{ \sing_1 N\lgth{s-1} 
\rightarrow 
N\lgth{s} \},
 \end{eqnarray*}
where the morphism is given by the Dyer-Lashof action,  since $\stein_0$ is the 
identity functor and $\stein_1 = \sing_1$. In particular, $(\lder_1 \indec 
N)\lgth{0}=0$ and $(\lder_0 \indec N)\lgth{0}=N\lgth{0}$.
\end{exam}

\begin{cor}
\label{cor:conn_lderq_N}
 Let $N \in \qmgr$ such that $\tau\lgth{\leq -1} N =0$ and $N\lgth{i}$ is $(d_i 
-1 \geq -2)$-connected for each $i$. 
 For $s, t \in \nat$, $(\lder_t q N )\lgth{s}$ is  at least $\Big(2^t (d_{s-t} +1) - 
(t+2)\Big)$-connected, hence is at least 
 $(2^t d_{s-t}-1)$-connected.  
\end{cor}

\begin{proof}
 A consequence of Corollaries \ref{cor:conn_stein} and 
\ref{cor:vanishing_lder_indec}.
\end{proof}

\section{Derived functors of $\indec$ and desuspension}
\label{sect:lderq_susp}
 
In this section, the behaviour of the functors $\lder_* \indec $ on a (de)suspension is considered.

\subsection{Relating $\lder_* \indec (\Sigma^{-1}M)$ and $\Sigma^{-1} \lder_* 
\indec M$}

Recall from Section \ref{sect:prelim} that the suspension morphism 
\[
 \epsilon_{\Sigma^{-1} M} : \sing \Sigma^{-1} M 
 \rightarrow 
 \Sigma^{-1} \sing M 
\]
is a natural transformation in $\qmgr$ for $M \in \amodgr$. This is compatible 
with the comonad structure of $\sing$: 

\begin{lem}
\label{lem:epsilon_comonad}
 \ 
 \begin{enumerate}
  \item 
  For $N \in \qmgr$ there is a natural commutative diagram in $\qmgr$:
  \[
   \xymatrix{
   \sing \Sigma^{-1} N 
   \ar[r]^{\epsilon_{\Sigma^{-1}N}}
   \ar[d]_{\mu_{\Sigma^{-1}N}}
   &
   \Sigma^{-1}  \sing  N 
       \ar[d]^{\Sigma^{-1}\mu_{N}}
       \\
       \Sigma^{-1}N
       \ar@{=}[r]
       &
       \Sigma^{-1}N,
   }
  \]
where $\mu$ denotes the adjunction counit.
\item  
For $M \in \amodgr$, the adjunction unit $\eta$ fits into a commutative diagram:
\[
 \xymatrix{
 \Sigma^{-1}M 
 \ar[r]^{\Sigma^{-1}\eta_M}
 \ar[d]_{\eta_{\Sigma^{-1}M}}
 &
 \Sigma^{-1}\sing M 
 \ar@{=}[d]
 \\
 \sing \Sigma^{-1} M 
 \ar[r]_{\epsilon_{\Sigma^{-1}M}}
 &
  \Sigma^{-1}\sing M .
 }
\]
\item 
For $N \in \qmgr$, there is a natural commutative diagram in $\qmgr$:
\[
 \xymatrix{
 \sing\Sigma^{-1}N 
 \ar[rr]^{\epsilon_{\Sigma^{-1}N}}
 \ar[d]_{\Delta_{\Sigma^{-1}N}}
 &&
 \Sigma^{-1}\sing N
 \ar[d]^{\Sigma^{-1} \Delta_N}
 \\
 \sing \sing \Sigma^{-1} N
 \ar[rr]_{\epsilon_{\Sigma^{-1}\sing N}\sing\epsilon_{ \Sigma^{-1}N}}
 &&
  \Sigma^{-1}\sing \sing N ,
 }
\]
in which $\Delta : \sing \rightarrow \sing \sing$ is the comonad structure map. 
 
 \end{enumerate}
\end{lem}

\begin{proof}
 The first two points are straightforward, using the fact that $\epsilon$ is 
constructed via the adjunction from $\eta$. The third then follows.
\end{proof}

\begin{prop}
 \label{prop:epsilon_resolutions}
For $N \in \qmgr$,  $\epsilon$ induces natural 
transformations:
 \begin{enumerate}
  \item 
  $
   \sing^{\bullet +1}  \Sigma^{-1} N
   \rightarrow 
   \Sigma^{-1} \sing^{\bullet +1}  N
  $ in $\ch \qmgr$;
\item 
$
 \qcx \Sigma^{-1}N 
 \rightarrow 
 \Sigma^{-1} \qcx N
$
 in $\ch \amodgr$;
\item 
\label{label:lder_epsilon}
$
 \lder_i \indec \Sigma^{-1}N 
 \rightarrow 
 \Sigma^{-1} \lder_i \indec N
$
 in $\amodgr$, 
for $i \in\nat$, which is an isomorphism for $i=0$.
 \end{enumerate}
\end{prop}

\begin{proof}
 The first statement follows from Lemma \ref{lem:epsilon_comonad} and the 
remaining statements follow since $\indec$ commutes with $\Sigma^{-1}$.
\end{proof}

\begin{rem}
 Analogous results hold in the non-length-graded case, replacing $\amodgr$ and 
$\qmgr$ respectively by $\amod$ and $\qm$.
\end{rem}

The morphism appearing in Proposition \ref{prop:epsilon_resolutions} 
(\ref{label:lder_epsilon}) fits into a long exact sequence 
associated to a Grothendieck composite functor spectral sequence. 

\begin{prop}
\label{prop:lder_indec_Gss}
 There is a natural isomorphism of functors $\indec \qsusp \cong \Sigma \indec 
: 
\qm \rightarrow \amod$ and the associated 
 Grothendieck spectral sequence 
 \[
  (\lder_s \indec ) \qsusp_t 
  \Rightarrow 
  \Sigma \lder_{s+t} \indec 
 \]
degenerates to a long exact sequence of the form 
\[
 \ldots 
 \rightarrow 
 (\lder_{n-1}\indec)\qsusp_1 
 \rightarrow 
 \Sigma \lder_n \indec 
 \rightarrow 
 (\lder_n \indec) \qsusp 
 \rightarrow  
 (\lder_{n-1}\indec)\qsusp_1 
\rightarrow \ldots .
 \]
The natural morphism $\Sigma \lder_n \indec 
 \rightarrow 
 (\lder_n \indec) \qsusp $ evaluated on $\Sigma^{-1} N$ yields the natural 
transformation of Proposition \ref{prop:epsilon_resolutions} 
(\ref{label:lder_epsilon}).
\end{prop}

\begin{proof}
 For the existence of the spectral sequence, it suffices to use the fact that 
$\qsusp$ carries relative projectives to relative projectives, which holds by 
Corollary \ref{cor:qsusp}.
  The remainder is straightforward.
\end{proof}

\begin{rem}
 This spectral sequence is analogous to that introduced by Miller (see 
\cite[Equation 2.3.3]{Miller} and following). This is defined for functors on a 
certain category of allowable Hopf algebras
  and is the composite of the indecomposables functors $Q$ and the functor 
$\untor_0 (\field, -)$. The functor $Q$ has $\lder_n Q=0$ for $n>1$ so the 
spectral sequence is of similar form.
\end{rem}

\subsection{The suspension and the Koszul complex $\kz$}

\begin{prop}
 \label{prop:susp_koszul_cx}
The suspension morphism $\stein_* \Sigma^{-1} \twoheadrightarrow \Sigma^{-1} 
\stein_*$  of Proposition \ref{prop:stein_sigma_surject} induces a surjective 
morphism 
of Koszul complexes for $N \in \qmgr$:
\[
 \kz \Sigma^{-1} N 
\twoheadrightarrow 
\Sigma^{-1} \kz N,
\]
for $s \in \nat$.

The induced morphism in homology, $$H_t  \kz \Sigma^{-1} N 
\rightarrow 
H_t \Sigma^{-1} \kz N$$ 
identifies via the isomorphism of Corollary \ref{cor:vanishing_lder_indec} with 
the natural transformation
\[
   \big( \lder_t \indec \Sigma^{-1} N) \lgth{s} 
\rightarrow 
  \big( \Sigma^{-1} \lder_t \indec N) \lgth{s} 
\]
 of Proposition \ref{prop:epsilon_resolutions}.
\end{prop}

\begin{proof}
For the first statement, it suffices to check that the differential of the 
Koszul complexes is compatible with the suspension morphisms
 $\stein_* \Sigma^{-1} \twoheadrightarrow \Sigma^{-1} \stein_*$. This follows 
from the commutativity of the following
diagram, in which the horizontal morphisms are the coproducts of Proposition 
\ref{prop:Steinberg} and the vertical morphisms
 the suspension morphisms of Proposition \ref{prop:stein_sigma_surject}:
\[
 \xymatrix{
\stein_{s+t} \Sigma^{-1} 
\ar[dd]
\ar[r]
&
\stein_s \stein_t \Sigma^{-1}
\ar[d]
\\
&
\stein_s \Sigma^{-1} \stein_t 
\ar[d]
\\
\Sigma^{-1}\stein_{s+t}
\ar[r]
&
\Sigma^{-1} \stein_s  \stein_t 
}
\]
where $s,t \in \nat$.

The identification of the induced morphism in homology follows from unravelling the 
definitions.
\end{proof}

\begin{rem}
 By Proposition \ref{prop:ses_stein_susp}, the kernel of the surjective morphism 
of complexes 
$ \kz \Sigma^{-1} N 
\twoheadrightarrow 
\Sigma^{-1} \kz N$ is a complex with $i$th term 
\[
 \Sigma^{-1} \stein_{i-1} \Sigma^{-1} \Phi N
\]
 and with differential induced from $\kz \Sigma^{-1}N$. 
\end{rem}

\section{The functor $H_0 \rcx$ and destabilization}
\label{sect:destab}

The relationship with instability of $\cala$-modules is established by the construction of a 
chain complex that calculates 
the derived functors of destabilization. 
Such results go back to Singer (see \cite{Singer} for example), the work of 
Lannes and Zarati \cite{LZ} on 
derived functors of destabilization 
and more recent treatments such as \cite{p_destab} (working with cohomology and for odd primes) 
and \cite{KMcC}. The latter  reference is 
followed here, since the relationship with the action of Dyer-Lashof operations 
is fundamental. 

\subsection{The chain complex}
\label{subsect:chcx}

This section presents the construction of the chain complex of \cite{KMcC} in a 
form suitable for current purposes.

\begin{nota}
 For $M \in \amod$, let $d_M : M \rightarrow \sing_1 \Sigma M \subset \sing 
\Sigma M$ denote the Singer differential (see \cite[Definition 4.11]{KMcC}) 
and also its extension to a morphism of 
$\qm$:
\[
 d_M : \sing M \rightarrow \sing \Sigma M.
\]
\end{nota}
 
The following resumes some of the fundamental properties of $d_M$:

\begin{prop}
 \cite{KMcC} 
For $M \in \amod$,
\begin{enumerate}
 \item 
the kernel of $\Sigma d_{\Sigma^{-1} M} : M \rightarrow \Sigma \sing_1 M$ is 
$\Omega^\infty M \subset M$;
\item 
the composite $d_{\Sigma M} d_M : \sing M \rightarrow \sing \Sigma^2 M$ is 
trivial;
\item 
$d$ commutes with the suspension $\epsilon : \sing \rightarrow \Sigma^{-1} 
\sing \Sigma$.
\end{enumerate}
\end{prop}

The above statements  can be made more precise by observing that the 
differential respects the length grading, increasing it by 
one, hence can be written as the natural transformation of 
$\qmgr$: 
\[
 d_M : \sing M \rightarrow \big( \sing \Sigma M \big) (-1),
\]
for $M \in \amod$, using the notation of Section \ref{subsect:length1}.

\begin{defn}
\label{def:rcx}
For $M \in \amod$, let $\rcx M$ denote the homological chain complex in $\qmgr$ with 
$$\rcx_s M := \big(\sing \Sigma^{-s} M\big)(s)$$
and differential $d_{\Sigma^{-s} M}(s) : \rcx_s 
M \rightarrow \rcx_{s-1} M$.
\end{defn}

\begin{nota}
\ 
\begin{enumerate}
 \item 
Let $\ch \qmgr$ denote the category of $\zed$-graded, homological chain complexes in $\qmgr$. 
\item 
For $d\in \zed$, let $[d] : \ch \qmgr \rightarrow \ch \qmgr$ denote the $d$-fold 
homological shift of complexes.  
\end{enumerate}
\end{nota}

\begin{prop}
\label{prop:rcx_length}
 The construction $\rcx$ defines  an exact functor $\rcx : \amod \rightarrow 
\ch 
\qmgr$. 
This satisfies the following properties for $M \in \amod$:
\begin{enumerate}
 \item 
$H_0 \rcx M \cong \bigoplus _{i\geq 0} \Sigma^{-1} \Omega^\infty _i 
\Sigma^{1-i} M  \in \qmgr$;
\item 
for $d\in \zed$, there is a natural isomorphism $\rcx \Sigma^d M \cong \big(\rcx M \big)[d](d)$ 
and 
$H_d \rcx M \cong H_0 \rcx \Sigma^{-d} M (d)$;
\item 
there is an exact sequence of complexes in $\qmgr$:
\[
 0
\rightarrow 
\big(\Phi \rcx M \big)(1) 
\stackrel{Q_0}{\rightarrow}
\rcx M
\stackrel{\epsilon}{\rightarrow}
\Sigma^{-1} \rcx \Sigma M
\rightarrow 
0 .
\]
\end{enumerate}
\end{prop}

\begin{proof}
 The first statement is a consequence of \cite[Theorem 4.22]{KMcC} and the second
is clear from the definition of the chain complex; the final statement is a chain complex version of Corollary \ref{cor:qsusp} 
(cf. \cite[Proposition 4.26]{KMcC}).
\end{proof}

\begin{prop}
\label{prop:H0_unstable}
For $M \in \unst$ unstable, there is a natural isomorphism $H_0 \rcx M \cong 
\sing M$ in $\qmgr$.
\end{prop}

\begin{proof}
 This fundamental fact relies on the calculation of certain derived functors of 
destabilization by Lannes and Zarati
\cite{LZ}, which can be recovered from the chain complex viewpoint as in 
\cite[Theorem 4.34]{KMcC} (see also \cite{p_destab}).
Namely, for $M$ unstable, $\Omega^\infty_i \Sigma ^{1-i}M \cong \Sigma \sing_i 
M$ and the proof shows that 
  $\bigoplus _{i\geq 0} \Sigma^{-1} \Omega^\infty _i \Sigma^{1-i} M \cong \sing 
M$ in $\qmgr$ (i.e. the Dyer-Lashof action is the canonical one).
\end{proof}

The following connectivity result is useful:

\begin{cor}
\label{cor:connectivity_H0rcx}
If $M \in \amod$ is $c$-connected, then for $i \in \nat$, $(H_0 \rcx 
M)\lgth{i}$ 
is at least $(2^i (c+1) -1)$-connected.  

If $c \geq -1$ and for $l \in \nat$,  $\big( \lder_* \indec H_0 \rcx M \big)\lgth{l}$ is 
at least $(2^l(c+1) -1)$-connected.
\end{cor}

\begin{proof}
Using the identification of  $H_0 \rcx M \cong \bigoplus _{i\geq 0} \Sigma^{-1} 
\Omega^\infty _i \Sigma^{1-i} M $ from  Proposition \ref{prop:rcx_length}, 
this result follows from 
\cite[Corollary 4.31]{KMcC} (and can be proved directly from the connectivity 
statement in Proposition \ref{prop:properties_sing}).

The final statement then follows by applying Corollary \ref{cor:conn_lderq_N}.
\end{proof}

\subsection{The long exact sequence for $H_0 \rcx$}
The exactness of $\rcx : \amod \rightarrow \ch \qmgr$ leads to the following:

\begin{prop}
\label{prop:les_H0rcx}
For $0\rightarrow M_1 \rightarrow M_2 \rightarrow M_3 \rightarrow 0$ a short 
exact sequence in $\amod$, there is a natural long exact sequence in $\qmgr$:
\[
\ldots
\rightarrow 
 H_0 \rcx \Sigma^{-1} M_3 (1) 
 \rightarrow 
 H_0 \rcx M_1 
 \rightarrow 
 H_0 \rcx M_2 
 \rightarrow 
 H_0 \rcx M_3 
 \rightarrow 
 H_0 \rcx \Sigma M_1 (-1)
 \rightarrow 
 \ldots 
 .
\]
\end{prop}

\begin{proof}
 This is the long exact sequence in homology associated to the short exact 
sequence in $\ch \qmgr$: 
 \[
  0 
  \rightarrow 
   \rcx M_1 
   \rightarrow 
   \rcx M_2 
   \rightarrow 
   \rcx M_3
   \rightarrow 0,
 \]
by using the natural isomorphisms $H_1 \rcx M_3 \cong H_0 \rcx\Sigma^{-1} 
M_3(1) 
$ and $H_{-1} \rcx M_1 \cong  H_0 \rcx \Sigma M_1 (-1)$ provided by Proposition 
\ref{prop:rcx_length} (the non-displayed terms are treated similarly).
\end{proof}

\subsection{Connectivity estimates}

For $M \in \amod$ and $c \in \zed$, there is a natural truncation short exact 
sequence 
\[
 0
 \rightarrow 
 M_{\leq c} 
 \rightarrow 
 M
 \rightarrow 
 M_{\geq c+1}
 \rightarrow 
 0
\]
in which $M_{\leq c}$ is the subobject of $M$ formed by elements of degree at 
most $c$ so that $M_{\geq c+1}$ is at least $c$-connected.

\begin{prop}
\label{prop:conn_H0rcxM}
 For $M\in \amod$, $c \in \zed$ and $i \in \nat$, the natural morphism 
 \[
 \big( H_0 \rcx (M_{\leq c})\big)\lgth{i}
  \rightarrow 
  \big( H_0 \rcx M\big)\lgth{i}
 \]
is 
\begin{enumerate}
 \item 
 surjective in degrees $\leq 2^i (c+1) -1$;
 \item 
 injective if $i=0$ and, if $i>0$, in degrees $\leq 2^{i-1}c -1$. 
\end{enumerate}
In particular, for $c \in \nat$, it is an isomorphism in degrees 
\begin{enumerate}
 \item 
 $\leq c$ for $i=0$; 
 \item 
$\leq  2^{i-1}c -1$ for $i>0$.
\end{enumerate}
\end{prop}

\begin{proof}
 The result follows by applying the connectivity result of Corollary 
\ref{cor:connectivity_H0rcx} to the exact sequence 
 \[
 \big( H_0 \rcx \Sigma^{-1} (M _{\geq c+1})\big) \lgth{i-1}
  \rightarrow 
 \big( H_0 \rcx (M_{\leq c}) \big)\lgth{i}
 \rightarrow
 \big(H_0 \rcx M \big)\lgth{i}
 \rightarrow
 \big(H_0 \rcx (M_{\geq c+1}) \big)\lgth{i}
 \]
given by Proposition \ref{prop:les_H0rcx}.
\end{proof}

\begin{cor}
\label{cor:conn_lder_q_HO}
 For $M \in \amod$ that is $(-1)$-connected and $0< c, l \in \nat$, the 
 morphism 
 \[
  \big( \lder_* \indec H_0 \rcx M_{\leq c} \big )\lgth{l}
  \rightarrow 
  \big( \lder_* \indec H_0 \rcx M \big )\lgth{l}
 \]
induced by the canonical inclusion $M_{\leq c} \hookrightarrow M$ is an 
isomorphism in degrees $\leq 2^{l-1} c -1$. 
\end{cor}

\begin{proof}
 The result follows by combining Corollary \ref{cor:conn_lderq_N} 
with 
Proposition \ref{prop:conn_H0rcxM}.
\end{proof}

\section{The suspension morphism for $H_0 \rcx$}
\label{sect:H0R_susp}

The behaviour of $H_0 \rcx$ with respect to the suspension morphism is of 
independent interest; the material of this section
sheds light on results of Kuhn and McCarty \cite{KMcC}.

\subsection{The suspension morphism $\epsilon$}
\label{subsect:susp}

Recall from Proposition 
\ref{prop:rcx_length} that, for $M \in \amod$, the suspension morphism induces a 
morphism of 
complexes in $\qmgr$
\[
 \epsilon : \rcx M \rightarrow \Sigma^{-1} \rcx \Sigma M.
\]
The morphism $H_0 \epsilon$ has adjoint $\qsusp 
  H_ 0 \rcx M
  \rightarrow 
  H_0 \rcx \Sigma M $.

\begin{thm}
\label{thm:susp_rcx_H0}
Let $M \in \amod$.
 \begin{enumerate}
  \item 
  \label{item:ses_susp_qsusp}
  There is a natural short exact sequence in $\qmgr$:
  \[
   0
   \rightarrow 
   \qsusp 
  H_ 0 \rcx M
  \rightarrow 
  H_0 \rcx \Sigma M 
  \rightarrow 
  \qsusp_1 H_0 \rcx \Sigma M (-1)
  \rightarrow 
  0.   
  \]
\item 
There are natural isomorphisms in $\qmgr$:
\begin{eqnarray*}
 \qsusp H_0 \rcx M &\cong & \mathrm{image} \{
   H_0\rcx M 
   \stackrel{H_0 \epsilon}{\longrightarrow}
   \Sigma^{-1} H_0 \rcx \Sigma M
   \}\\
   &\cong &  \bigoplus _{i\geq 0} \Omega \Omega^\infty _i \Sigma^{1-i} (\Sigma 
M).
\end{eqnarray*}
\item 
\label{item:ses_imQ}
Writing $\imQ$ for the image of $\Phi (H_0 \rcx M) (1) 
\stackrel{Q_0}{\rightarrow} H_0 \rcx M$, there are  natural short exact 
sequences in $\qmgr$:
\[
 0
 \rightarrow 
 \imQ 
 \rightarrow 
 H_0 \rcx M
 \rightarrow 
 \Sigma^{-1} \qsusp H_0 \rcx M
 \rightarrow 
 0
\]
\[
 0 
 \rightarrow 
 \Sigma^{-1} 
 \qsusp_1 
 H_0 \rcx M
 \rightarrow 
 \Phi \big( H_0 \rcx M) (1) 
 \rightarrow 
 \imQ
 \rightarrow 
 0.
\]
\item 
There is a natural isomorphism in $\qmgr$
\[
 \imQ \cong  \Sigma^{-1}  \bigoplus_{i \geq 1} \Omega_1 
\Omega^\infty_{i-1}\Sigma^{1-i} (\Sigma M).
\]
\end{enumerate}
\end{thm}

\begin{proof}
The short exact sequence of complexes in $\qmgr$ given by Proposition 
\ref{prop:rcx_length},
\[
 0
\rightarrow 
\big(\Phi \rcx M \big)(1) 
\stackrel{Q_0}{\rightarrow}
\rcx M
\stackrel{\epsilon}{\rightarrow}
\Sigma^{-1} \rcx \Sigma M
\rightarrow 
0 ,
\]
induces a long exact sequence in homology, occurring as the top row of the 
following commutative diagram:
\[
 \xymatrix@C=-1.25pc @R=1.5pc{
\ldots\ar[r]
 &
 \Phi \big(H_0 \rcx M\big)(1) 
 \ar[rr]^(.6){Q_0}
 \ar@{.>>}[rd]
 &&
 H_0 \rcx M 
  \ar@{.>>}[rd]
  \ar[rr]^{H_0 \epsilon} 
 &&
 \Sigma^{-1} H_0 \rcx \Sigma M 
 \ar[rr]
   \ar@{.>>}[rd]
   &&
   \Phi \big(H_{-1} \rcx M\big)(1)
 \\
 \Sigma^{-1} \qsusp_1 H_0 \rcx M
 \ar@{^(.>}[ur]
 &&
\imQ 
 \ar@{^(.>}[ur]
 &&
\Sigma^{-1} \qsusp H_0 \rcx M
 \ar@{^(.>}[ur]
 &&
  \Sigma^{-1} \qsusp_1 H_{-1} \rcx M
   \ar@{^(.>}[ur]
   }
 \]
 where the exactness of the functors $\Sigma^{-1}$, $\Phi$ and $\cdot(1)$ 
provides the 
commutation with formation of homology and 
  the identification of the kernel and  cokernel of $Q_0$ comes from 
Proposition 
\ref{prop:qsuspgr}.
    
  This provides the short exact sequence of part (\ref{item:ses_susp_qsusp}), 
by 
using the isomorphism $H_{-1} \rcx M \cong H_0 \rcx \Sigma M (-1)$ given 
  by Proposition \ref{prop:rcx_length}. The identification of $\qsusp H_0 \rcx 
M$ follows using the argument employed in the proof of \cite[Proposition 
1.11]{KMcC}.

The two short exact sequences of (\ref{item:ses_imQ}) are given  by 
dévissage of the long exact sequence, as indicated by the dotted arrows in the 
diagram. 

Finally, in length grading $i$, the short exact sequence  
$$0
 \rightarrow 
 \imQ \lgth{i}
 \rightarrow 
 H_0 \rcx M\lgth{i}
 \rightarrow 
 \Sigma^{-1} \qsusp H_0 \rcx M\lgth{i}
 \rightarrow 
 0$$
 identifies as 
 \[
  0
  \rightarrow 
  \big(\imQ\big)\lgth{i}
  \rightarrow 
  \Sigma^{-1} \Omega^{\infty}_i \Sigma^{1-i} M 
  \rightarrow 
  \Sigma^{-1} \Omega \Omega^{\infty}_i \Sigma^{1-i}\Sigma M 
  \rightarrow 
  0,
 \]
where the surjection is the desuspension of the surjection derived from the 
 short exact sequence of Proposition \ref{prop:Omega_infty_desusp}:
\[
 0 
 \rightarrow 
 \Omega_1 \Omega^\infty_{i-1} \Sigma 
 \rightarrow 
 \Omega^\infty_i
 \rightarrow 
 \Omega\Omega^\infty _i \Sigma
 \rightarrow 
 0
\]
applied to $\Sigma^{1-i}M$. This yields the identification of $\imQ$. 
\end{proof}

\begin{exam}
\label{exam:L*M}
 In \cite[Proposition 1.11]{KMcC}, for $M \in \amod$, Kuhn and McCarty 
introduce 
an object 
$L_* M$ that lies in $\qmgr$ and such that each component is unstable. This is 
used in their description
of the $E^\infty$-page of the algebraic approximation to the infinite looping 
spectral sequence (see \cite[Theorem 1.12]{KMcC}). 

This object is naturally isomorphic to $\qsusp H_0 \rcx \Sigma^{-1}M$, where 
the 
passage to $\qsusp$
 serves to ensure $\cala$-instability. (The notation $L_*M$ is not used here, 
since it could lead to 
confusion with the Steinberg functors.)
\end{exam}

\begin{cor}
\label{cor:delooping_H0RM}
 For $M \in \amod$, there is a natural exact sequence in length-graded objects 
of $\Sigma \qm \cap \unst$:
 \[
  0
  \rightarrow 
  \Sigma 
  (\qsusp H_0 \rcx M)
  \rightarrow 
  (\Sigma H_0 \rcx \Sigma M) 
  \stackrel{Sq_0} {\rightarrow}
  \Phi(\Sigma H_0 \rcx \Sigma M) 
  \rightarrow 
  \Sigma^2 \imQ[\Sigma M] (-1)
  \rightarrow 
  0
 \]
that identifies in each length grading with the exact sequence of Proposition 
\ref{prop:Omega_es}.
\end{cor}

\begin{proof}
 The suspensions ensure that each term lies in $\unst$, hence $Sq_0$ is a 
morphism of $\qmbig$ by Proposition \ref{prop:Sq_0_DL-linear}. 
 The terms are identified by Theorem \ref{thm:susp_rcx_H0} (by reassembling two 
of the short exact sequences). 
\end{proof}

\begin{rem}
 \ 
 \begin{enumerate}
  \item 
  The previous results should be compared with \cite[Corollary 4.36]{KMcC}.
\item 
Corollary \ref{cor:delooping_H0RM} shows what is required to pass from $H_0 
\rcx 
M$ to $H_0 \rcx \Sigma M$; whereas $\qsusp H_0 \rcx M$ is calculated 
as a functor of $H_0 \rcx M$, one also requires the input of $\imQ[\Sigma M]$. 
With these in hand, the delooping spectral sequence 
of Theorem \ref{thm:delooping_ss} can be used, enhanced so as to take into 
account the Dyer-Lashof action. (Compare Example \ref{exam:imQ_trivial} below.)
 \end{enumerate}
\end{rem}

\begin{cor}
\label{cor:suspension_stable_range}
For $M \in \amod$ which is $c$-connected and $i \in \nat$, the following 
statements hold.
 \begin{enumerate}
  \item 
The natural inclusion
 \[
  \big(
  \qsusp H_0 \rcx M
  \big)\lgth{i}
  \hookrightarrow 
  \big(
  H_0 \rcx \Sigma M
  \big)\lgth{i}
 \]
is an isomorphism in degrees $\leq 2^{i+1} (c+2)$.
\item 
The natural surjection 
\[
\big (H_0 \rcx M \big) \lgth{i}
\twoheadrightarrow
\big(
 \Sigma^{-1}
  \qsusp H_0 \rcx M
  \big)\lgth{i}
\]
is an isomorphism for $i =0$ and, for $i>0$, in  degrees $\leq 2^i (c+1) -2$.
\end{enumerate}
In particular, if $c \geq -1$, the suspension morphism 
\[
 \big(H_0 \rcx M\big)\lgth{i}
 \stackrel{H_0 \epsilon}{\rightarrow }
 \big( \Sigma^{-1} H_0 \rcx \Sigma M \big) \lgth{i}
\]
is an isomorphism in degrees $\leq 2^i (c+1)-2$.
\end{cor}

\begin{proof}
For the first statement, consider the short exact sequence 
  \[
   0
   \rightarrow 
\big(   \qsusp H_ 0 \rcx M \big)\lgth{i}
  \rightarrow 
\big(  H_0 \rcx \Sigma M  \big)\lgth{i}
  \rightarrow 
 \big( \qsusp_1 H_0 \rcx \Sigma M \big)\lgth{i+1}
  \rightarrow 
  0
  \]
of Theorem \ref{thm:susp_rcx_H0}.  Now  $\big(H_0 \rcx \Sigma M \big)\lgth{i}$ 
is at least $(2^i (c+2)-1)$-connected, by Corollary 
\ref{cor:connectivity_H0rcx}, hence Proposition \ref{prop:qsusp_1_conn} implies 
that $\big( \qsusp_1 H_0 \rcx \Sigma M \big)\lgth{i+1}$ is at least $2^{i+1} 
(c+2)$-connected, for $i \in \nat$.

The second point is proved by a similar argument, based upon the connectivity 
of 
$\Phi H_0 \rcx M (1)$ and hence of the image of $\imQ$. The conclusion for 
$H_0 \epsilon$ follows by putting these statements together.
\end{proof}

\begin{exam}
\label{exam:imQ_trivial}
 Suppose that $\imQ =0$, then one has 
 \[
  \qsusp_1 H_0 \rcx M 
  \cong 
  \Sigma \Phi H_0 \rcx M ( 1).
 \]
In this case, the short exact sequence of Theorem \ref{thm:susp_rcx_H0} 
part (\ref{item:ses_susp_qsusp}) for $\Sigma^{-1} M$ becomes:
 \[
   0
   \rightarrow 
   \qsusp 
  H_ 0 \rcx \Sigma^{-1} M
  \rightarrow 
  H_0 \rcx  M 
  \rightarrow 
  \Sigma \Phi H_0 \rcx M 
  \rightarrow 
  0   
  \]
  in $\qmgr$.
  
  If, moreover, $Q_0$ is trivial on $H_0 \rcx \Sigma^{-1}M$, then  $\qsusp 
  H_ 0 \rcx \Sigma^{-1} M \cong \Sigma H_ 0 \rcx \Sigma^{-1} M$. After 
suspension, the short exact sequence then becomes:
  \[
    0
   \rightarrow 
   \Sigma (
  \Sigma H_ 0 \rcx \Sigma^{-1} M ) 
  \rightarrow 
  (\Sigma H_0 \rcx  M) 
  \rightarrow 
  \Phi (\Sigma H_0 \rcx M)
  \rightarrow 
  0.  
  \]
Here $(\Sigma H_ 0 \rcx \Sigma^{-1} M )$ and $  (\Sigma H_0 \rcx  M) $ are 
length-graded unstable modules and 
the surjection $(\Sigma H_0 \rcx  M) 
  \rightarrow 
  \Phi (\Sigma H_0 \rcx M) $ is induced by $Sq_0$, with surjectivity a 
consequence of the vanishing of $Q_0$. 
  
  In this case, it follows that there is an identification
  \[
   \Omega  (\Sigma H_0 \rcx  M) \cong (  \Sigma H_ 0 \rcx \Sigma^{-1} M ).
  \]
  
  This situation arises for example when considering the case $M = \Sigma^n 
\cala^* $ for $n>0$ an integer (cf. Example \ref{exam:dual_Steenrod}).
\end{exam}

\section{An infinite delooping spectral sequence}
\label{sect:ss}

The spectral sequence constructed in this section is derived from a double 
complex. The construction is analogous to 
that of the Grothendieck spectral sequence associated to a composite of 
functors: one of the spectral sequences associated to the 
double complex degenerates and the second gives the required spectral 
sequence.

\subsection{Construction of the spectral sequence}

In Section \ref{sect:destab} introduced the exact functors 
$
 \rcx : \amod \rightarrow \ch \qmgr
$
and $\qcx : \qmgr \rightarrow 
\ch \amodgr$. The composite $\qcx \rcx$ defines an exact functor 
to homological 
bicomplexes in $\amodgr$, 
where 
\[
 (\qcx \rcx M) _{s,t} = (\qcx)_t (\rcx_s M)
\]
so that the bicomplex is concentrated in bidegrees $(s, t) \in \zed \times 
\nat$ 
and the terms of the bicomplex belong to $\amodgr$ (i.e. are length-graded 
$\cala$-modules).
In particular, this can be considered as a length-graded bicomplex in $\amod$. 
Because of the 
shifting properties given in Proposition \ref{prop:rcx_length}, without loss of 
generality, we may restrict to
length grading zero: 
\[
  (\qcx \rcx M)\lgth{0}.
\]

\begin{lem}
 For $M \in \amod$, $(\qcx \rcx M)\lgth{0}$ is a second quadrant, 
homological bicomplex (concentrated in the quadrant $s \leq 0$, $t\geq 0$).
\end{lem}

\begin{proof}
 An immediate consequence of the fact that $\rcx_s M$ for $s>0$ is concentrated 
in positive length degree.
\end{proof}

\begin{rem}
\label{rem:boundedness}
If $M \in \amodc$, Proposition \ref{prop:properties_sing} provides a local 
finiteness property of this bicomplex: for  $n, d\in \zed$, 
$$\big((\qcx)_{n-s} (\rcx_s M)\big)_d =0$$ for $s \ll 0$. This is 
sufficient to ensure strong convergence of the spectral sequences considered 
below.
\end{rem}

There are two spectral sequences associated to the horizontal and vertical 
filtrations of the bicomplex.
The horizontal filtration is the decreasing filtration defined by 
\[
(\qcx \rcx M)\lgth{0}_{s \leq -n}
\]
for $n \in \nat$, so that $(\qcx  \rcx M)\lgth{0}_{s \leq 0} = (\qcx 
\rcx M)\lgth{0}$.

Since  $\rcx_s M = \sing \Sigma^{-s} M (s)$, which is acyclic for $\indec$, the 
associated spectral sequence has
\[
 E^1_{s,t} =\left\{ \begin{array}{ll}
                     M  & (s,t)=(0,0) \\
                     0 & \mathrm{otherwise}
                    \end{array}
\right.
\]
and degenerates at $E^1$.

The vertical filtration is the increasing filtration defined by 
\[
(\qcx \rcx M)\lgth{0} _{t\leq n}
\]
for $n \in \nat$, so that $(\qcx  \rcx M)\lgth{0} _{t\leq 0} = (\rcx 
M)\lgth{0}$.  

\begin{nota}
For $n \in \nat$, denote by $\vertic_n M \subset M$ the image in homology of 
the 
morphism of the total complexes associated to $(\qcx \rcx M)\lgth{0} 
_{t\leq n} \hookrightarrow (\qcx  \rcx M)\lgth{0}$
\end{nota}

\begin{thm}
\label{thm:ss}
 Let $M \in \amodc$.
\begin{enumerate}
 \item 
There is a natural second quadrant homological spectral sequence  $(E^r_{s,t}, 
d^r)$
 with $d_r$ of $(s,t)$-bidegree $(r-1, -r)$ and $E^2$-page: 
\[
 E^2_{s,t} =
\lder_t\indec \Big(
H_0 \rcx \Sigma^{-s}M
\Big)\lgth{-s}.
\]
\item 
There is a natural isomorphism
\[
 E^2_{-t, t} 
 \cong 
 \ker 
 \big \{
 \stein_t \Sigma^{-1} \Omega^\infty \Sigma^{1+t}M 
 \rightarrow 
 \stein_{t-1} \Sigma^{-1} \Omega^{\infty}_1 \Sigma^{t}M 
 \big \}.
\]
 \item 
The $t=0$ edge morphism identifies with the natural inclusion
\[ 
\vertic_0 M \cong 
\Sigma^{-1} \Omega^{\infty} \Sigma M  
\hookrightarrow 
 M
\]
in $(s,t)$-degree $(0,0)$ and is zero elsewhere.
\item
$E^2_{s,t}=0$ for $t>-s$, so that the spectral sequence 
is concentrated below the anti-diagonal and there is a second edge homomorphism:
\[
 E^\infty_{-t, t} \cong \vertic_t M / 
\vertic_{t-1} M
\hookrightarrow 
 E^2_{-t, t}.
 \]
\item 
The spectral sequence converges strongly to  $M$; 
$E^\infty_{s,t}=0$ if $s \neq -t$ and $E^\infty_{-t, t} \cong \vertic_t M / 
\vertic_{t-1} M$.
\end{enumerate}
\end{thm}

\begin{proof}
 This is the second spectral sequence associated to the bicomplex  $(\qcx 
\rcx M)\lgth{0}$. Since the functor $\qcx$ is 
exact, 
  $$E^1_{s,t} \cong \Big((\qcx)_t  H_s \rcx M\Big)\lgth{0}$$
  and the differential $d^1$ is that of $\qcx$. This identifies the $E^2$-page 
of the spectral sequence, by using the isomorphism $H_s \rcx M \cong H_0 \rcx 
\Sigma^{-s} M (s)$
 provided by Proposition \ref{prop:rcx_length}. Identification of the edge 
morphism is straightforward.

The vanishing above the anti-diagonal for $s\leq -2$ follows from Corollary 
\ref{cor:vanishing_lder_indec} which also gives the identification of 
$E^2_{-t,t}$. 
Since the spectral sequence is homological, the anti-diagonal provides a second 
edge homomorphism.
 
Strong convergence, under the hypothesis that $M \in \amodc$, is a consequence 
of Remark \ref{rem:boundedness} and the associated filtration of $M$ is given 
by 
the construction of the spectral sequence.
\end{proof}

\begin{cor}
\label{cor:second edge hom}
 For $M \in \amodc$, the edge morphism $E^\infty_{-1,1} 
\hookrightarrow E^2_{-1, 1}$ is an isomorphism, 
 hence 
 \begin{eqnarray*}
  \vertic_1 M/ \vertic_0 M
  &\cong &  \ker 
  \{
\stein_1 \Sigma^{-1} \Omega^\infty \Sigma^2 M 
\rightarrow 
\Sigma^{-1} \Omega^{\infty}_1 \Sigma M
  \}.
 \end{eqnarray*}
\end{cor}

\begin{proof}
The terms of $E^2_{-1, 1}$ are permanent cycles and $E^2_{-1,1}$ is given by 
Example \ref{exam:lder_q_one}.
\end{proof}

\begin{cor}
 \label{cor:almost unstable}
 For $M \in \amodc$ and $t \in \nat$, 
 the module $\vertic_t M \subset M$ admits a finite filtration such that the 
associated graded satisfies:
 \[
  \Sigma^{2^t} \gr\big( \vertic_t M\big) \in \unst
 \]
is unstable.
\end{cor}

\begin{proof}
For $n \in \nat$,  the subquotient $\vertic_n M/ \vertic_{n-1} M $ is a 
submodule of $\stein_n 
\Sigma^{-1} \Omega^\infty \Sigma^{1+n} M$ by Theorem \ref{thm:ss}, hence the 
result 
 follows by applying Corollary \ref{cor:sing_stein_instability}.
\end{proof}

\begin{rem}
\label{rem:attentive}
 The attentive reader will observe that the $E^2$-page of the spectral sequence 
of Theorem \ref{thm:ss} depends upon $\Sigma^{-1} \Omega^{\infty} 
\Sigma^{t+1} M$ for all $t\in \nat$. 
 Since  any bounded-below $\cala$-module $M$ can be recovered as
  \[
   M \cong \mathrm{colim}_d  \Sigma^{-d} \Omega^{\infty} \Sigma^d M, 
  \]
  the input would appear to contain the information to be calculated. That the 
spectral sequence {\em does} provide an effective tool for calculations is 
explained by the connectivity results 
  of the following section, in particular Theorem \ref{thm:convergence}.
\end{rem}

\subsection{Connectivity results}
\label{subsect:conn_ss}

\begin{prop}
\label{prop:conn_estimate_ss}
Let $M \in \amodc$ be $c$-connected and consider the spectral sequence of 
Theorem \ref{thm:ss}.
Then $E^2_{-t, t}$ has connectivity at least
\[
 2^t (d(M)+1) - (t+2)  \geq 2^t d(M)-1
\]
where $d(M)= \sup \{ (c+t+1), -1 \}$.

If $c+t+1 \leq -1$, then $E^2_{-t,i}$, for $0 \leq i \leq t$ has connectivity at 
least $-(i+2)$.  
\end{prop}

\begin{proof}
For the first statement, by Corollary \ref{cor:vanishing_lder_indec}, it 
suffices to consider the connectivity 
of $\stein_t \Sigma^{-1} \Omega^{\infty} \Sigma^{1+t}M$. The module $\Sigma^{-1} 
\Omega^{\infty} \Sigma^{1+t}M$ has lowest class in degree 
at least $d(M)$, using the fact that $\Omega^{\infty} \Sigma^{1+t}M$ is 
unstable, hence is always $(-1)$-connected. The result then follows from 
Proposition \ref{prop:conn_stein}. 

The second statement is proved by a similar argument.
\end{proof}

\begin{rem}
This leads to an explicit form of the convergence stated in Theorem 
\ref{thm:ss}, by replacing Remark \ref{rem:boundedness} by the 
result of Proposition \ref{prop:conn_estimate_ss}. In particular, as soon as 
$c+t 
>0$, this ensures exponential growth of the connectivity with $t$.
\end{rem}

\begin{exam}
\label{exam:Sigma_n_cala}
 Consider the fundamental example $M= \Sigma^{n} \cala^*$, with $n \in \zed$ 
(cf. Example \ref{exam:dual_Steenrod}). Then $\Omega^\infty _i \cala^* =0$ for 
$i >0$ and $\Omega^{\infty} \Sigma^{t+1} M = \Omega^\infty \Sigma^{n+t+1}\cala$ 
(for $t\in \nat$) is $0$-connected unless $1+t+n=0$. In particular, for 
$E^{2}_{-t, t}$ to be non-zero in negative degree, this requires $n<0$.  

For $n<0$, the only terms in the spectral sequence of negative degree arise 
from 
 \[
E^2_{n+1,-(n+1)} = \stein_{-(n+1)} \Sigma^{-1} \field.
 \]
In particular, these are permanent cycles in the spectral sequence and 
contribute to the negative degree part of $\Sigma^{n} \cala$. 
By Proposition \ref{prop:conn_stein}, the lowest degree non-trivial class in 
$\stein _{-(n+1)} \Sigma^{-1} \field$ is in degree $-(-(n+1) +1 ) = n$, 
as expected. This is again a manifestation of Koszul duality (cf. Section 
\ref{subsect:Koszul}).
\end{exam}

Proposition \ref{prop:conn_estimate_ss} gives a generalized {\em stable range} 
for the increasing 
filtration $\vertic_n M$ of $M$.

\begin{cor}
\label{cor:filt_semistable}
 Let $M \in \amodc$ be $c$-connected and $n \in \nat$.
Then 
 \[
  \vertic_n M \hookrightarrow M
 \]
is an isomorphism in degrees 
$
 \leq 2^{n+1} (d(M)+1) - (n+3), 
$
where $d(M)= \sup \{ (c+n+2), -1 \}$.
\end{cor}

Forgetting the $\cala$-action, the spectral sequence splits into degree summands 
(as for the length grading), in particular
one can truncate and consider the behaviour in degrees $\leq D$, for an integer 
$D$. 
Then, combining Corollary \ref{cor:filt_semistable} with the connectivity 
result 
Corollary \ref{cor:conn_lder_q_HO} leads to the following:

\begin{thm}
\label{thm:convergence}
 Let $M \in \amodc$ be $(-1)$-connected and let $n \in \nat$. The spectral 
sequence of Theorem \ref{thm:ss} to calculate $M$ in degrees $\leq 2^{n+1}(n+1) 
-1$ 
 depends only upon 
 \[
  \tau\lgth{\leq t} H_0 \rcx \Sigma^t (M_{\leq c_t})
 \]
for $0 \leq t \leq n$, where 
\begin{eqnarray*}
 c_0 &=& 2^{n+1}(n+1) -1 \\
 c_t &=& 2^{n-t+2}(n+1) \mathrm{\ for \ } t>0.
\end{eqnarray*}
\end{thm}

\subsection{The suspension morphism and the spectral sequence}

The natural suspension morphism of complexes, $\epsilon : \rcx M \rightarrow 
\Sigma^{-1} \rcx \Sigma M$ induces a morphism of bicomplexes 
\[
 \qcx  \rcx M 
 \rightarrow 
 \qcx \Sigma^{-1} \Sigma M
\]
and hence, by composing with the natural transformation $\qcx \Sigma{-1} 
\rightarrow \Sigma^{-1} \qcx $ of Proposition \ref{prop:epsilon_resolutions}
and restricting to length grading zero, a morphism of bicomplexes 
\[
 \big( \qcx  \rcx M \big) \lgth{0}
 \rightarrow 
\Sigma^{-1} \big( \qcx \rcx \Sigma M \big)\lgth{0}.
\]

\begin{thm}
\label{thm:ss_susp}
 For $M \in \amodc$,
\begin{enumerate}
 \item 
the suspension morphism induces a morphism of spectral 
sequences 
\[
 E^r_{s,t} (M) \rightarrow \Sigma^{-1} E^r_{s,t} (\Sigma M);
\]
\item
given on the $E^2$-page by the natural transformation 
\[
 \big (\lder_t \indec H_0 \rcx \Sigma^{-s} M \big) \lgth{-s}
 \rightarrow 
 \Sigma^{-1} \big (\lder_t \indec H_0 \rcx \Sigma^{1-s} M \big) \lgth{-s}.
\]
provided by Proposition \ref{prop:epsilon_resolutions};
\item 
on $E^2_{-t,t}$, the morphism is induced by the natural transformation 
\[
 \stein_t \Sigma^{-1}  \Omega^\infty \Sigma^{1+t}M 
\rightarrow 
\Sigma^{-1} \stein_t \Sigma^{-1}  \Omega^\infty \Sigma^{2+t}M 
\]
induced by the natural transformation $\stein_t \Sigma^{-1} \rightarrow 
\Sigma^{-1} \stein_t$ of Proposition 
\ref{prop:stein_sigma_surject} together with the natural transformation induced 
by $\Omega^\infty \hookrightarrow \Sigma^{-1} \Omega^\infty \Sigma$;
\item 
the morphism induces a natural monomorphism of filtered objects
\[
 \vertic _t M \hookrightarrow \Sigma^{-1} \vertic_t \Sigma M
\]
which, for $t=0$, is the natural transformation
\[
 \Sigma^{-1} \Omega^\infty \Sigma M 
 \hookrightarrow 
 \Sigma^{-2} \Omega^\infty \Sigma^2 M .
\]
\end{enumerate}
 \end{thm}

\begin{proof}
 Straightforward.
\end{proof}

\subsection{A characterization of instability}

As an immediate application of the spectral sequence of Theorem \ref{thm:ss}, 
one has the following characterization of instability in terms of the $E^2$-page
 of the spectral sequence.

\begin{thm}
\label{thm:characterize_U}
For $M \in \amodc$, the following conditions are equivalent:
\begin{enumerate}
 \item 
 $\Sigma M \in \unst$.
 \item 
For all $t>0$ and $i>0$, $\big(\lder_t \indec H_0 \rcx \Sigma^i M\big)\lgth{i} 
=0$.
\item 
For all $t>0$, $\big(\lder_t\indec H_0 \rcx \Sigma^t M\big)\lgth{t}=0$.
\item 
For all $t>0$, the morphism from the Koszul complex 
\[
 \stein_t \Sigma^{-1} \Omega^\infty \Sigma^{1+t}M 
 \rightarrow 
 \stein_{t-1} \Sigma^{-1} \Omega^{\infty}_1 \Sigma^{t}M 
\]
is injective.
\end{enumerate}
\end{thm}

\begin{proof}
The equivalence of (3) and (4) follows from Corollary 
\ref{cor:vanishing_lder_indec}. 

The implication (1)$\Rightarrow$(2) follows from Propositions 
\ref{prop:H0_unstable} and \ref{prop:proj_class_acyclic} and 
(2)$\Rightarrow$(3) is clear. 

 The spectral sequence of Theorem \ref{thm:ss} provides the implication 
(3)$\Rightarrow$(1), since it implies that the $t=0$ edge morphism is an 
isomorphism 
in total $(s,t)$-degree zero.
\end{proof}

\section{Examples}
\label{sect:exam_ss}

Some insight into the behaviour of the spectral sequence of Theorem \ref{thm:ss} 
is gained by considering 
basic examples such as the case $M = \Sigma^{-2} N$, with $N \in \unst$ and the 
cases $M= \Sigma^n \cala^*$, for $n \in \zed$.

\subsection{The spectral sequence for desuspensions of unstable modules}
\label{subsect:Sigma_-2_unstable}

 Consider the spectral sequence of Theorem \ref{thm:ss} for $M= \Sigma^{-2}N$,  
where $N \in\unst$. This is not covered by Theorem \ref{thm:characterize_U}. 
Then $E^2_{0,0} = \Sigma^{-1}\Omega^\infty \Sigma^{-1} N \cong \Sigma^{-1} 
\Omega N$. To calculate the $E^2_{-1, *}$ column it 
suffices to consider $\tau\lgth{\leq 1} H_0 \rcx \Sigma^{-1}N$, this identifies 
as 
 \[
  \xymatrix{
  \Sigma^{-1}\Omega^\infty N = \Sigma^{-1} N
  \ar@/^1pc/@{.>}[r]
  &
  \Sigma^{-1} \Omega^{\infty}_1 \Sigma^{-1}N,
  }
 \]
where the dotted arrow indicates the action of Dyer-Lashof operations.

The short exact sequence given by Proposition \ref{prop:Omega_infty_desusp} for 
$\Omega^\infty_1 \Sigma^{-1}N$ identifies as
\begin{eqnarray*}
 0
 \rightarrow 
 \Omega_1 N
 \rightarrow 
 \Omega^\infty_1 \Sigma^{-1}N
 \rightarrow 
 \Omega \Omega^\infty _1 N
 \rightarrow 
 0
\end{eqnarray*}
and $ \Omega \Omega^\infty _1 N\cong \sing_1 N$, since $N$ is unstable. 

The inclusion 
$\Sigma^{-1} \Omega_1 N
 \rightarrow 
\Sigma^{-1} \Omega^\infty_1 \Sigma^{-1}N$
induces a natural short exact sequence in $\qmgr$:
\begin{eqnarray}
\label{eqn:Omega_infty_1_unstable}
 0
 \rightarrow 
 \Sigma^{-1} \Omega_1 N (1) 
 \rightarrow 
 \tau\lgth{\leq 1} H_0 \rcx \Sigma^{-1}N
 \rightarrow 
 \tau\lgth{\leq 1}\Sigma^{-1} \sing N 
 \rightarrow 
 0.
\end{eqnarray}
 
 Now, by Corollary \ref{cor:vanishing_lder_indec},  $\big(\lder_* \indec 
\Sigma^{-1}\sing 
N \big)\lgth{1}$ is calculated as the homology of the complex:
 \[
 \sing_1 \Sigma^{-1} N  
 \rightarrow 
 \Sigma^{-1} \sing_1 N ,
\]
 where the morphism is surjective, hence the non-zero homology corresponds to 
$\big(\lder_1 \indec \Sigma^{-1}\sing N\big)\lgth{1} $
  given by the kernel, which is isomorphic to $ \Phi \sing \Sigma^{-1} N$  by 
 the short exact sequence of Proposition \ref{prop:qmgr}.

The long exact sequence for $\lder_* \indec$ associated to the 
short exact sequence (\ref{eqn:Omega_infty_1_unstable}) induces an exact 
sequence:
 \[
 \big( \lder_1 \indec \Sigma^{-1}\Omega_1 N (1)\big)\lgth{1} 
  \rightarrow 
 \big( \lder_1 \indec H_0 \rcx \Sigma^{-1}N\big) \lgth{1}
 \rightarrow 
 \big( \lder_1 \indec \Sigma^{-1} \sing N \big) \lgth{1}
 \rightarrow 
 \Sigma^{-1} \Omega_1 N (1),
 \]
 where the left hand term is zero for length degree reasons.
 
Hence this identifies with the exact sequence 
\[
 0
 \rightarrow 
  \big( \lder_1 \indec H_0 \rcx \Sigma^{-1}N\big) \lgth{1}
  \rightarrow 
  \Phi \Sigma^{-1} N (1) 
  \rightarrow 
 \Sigma^{-1} \Omega_1 N (1)
\]
and the right hand morphism is induced by the canonical surjection $\Phi N 
\twoheadrightarrow \Sigma \Omega_1 N$ (recalling that $\Phi \Sigma^{-1} \cong 
\Sigma^{-2} \Phi$). 

Thus the $E^2$-page of the spectral sequence is closely related to the exact 
sequence of Proposition \ref{prop:Omega_es}
\[
 0
 \rightarrow 
 \Sigma \Omega N 
 \rightarrow 
 N
 \rightarrow 
 \Phi N 
 \rightarrow 
 \Sigma \Omega_1 N
 \rightarrow 
 0
\]
(after applying $\Sigma^{-2}$). In particular
\begin{eqnarray*}
 \vertic_0 (\Sigma^{-2} N) &\cong&\Sigma^{-1} \Omega N \\
 \vertic_1 (\Sigma^{-2} N)/ \vertic_0 (\Sigma^{-2} N) & \cong & \Sigma^{-2} 
\ker 
\{ \Phi N \rightarrow \Sigma \Omega_1 N \}
\\
\vertic_1 (\Sigma^{-2} N) & \cong & \Sigma^{-2} N.
\end{eqnarray*}

The suspension morphism of spectral sequences can also be described explicitly 
in this case, 
via the natural transformations
\[
 \vertic_t \Sigma^{-2}N 
\rightarrow 
\Sigma^{-1} \vertic_t \Sigma^{-1}N.
\]
Now, $M= \Sigma^{-1}N$ falls within the case of Theorem 
\ref{thm:characterize_U}, in particular 
$\vertic_t \Sigma^{-1}N = \Sigma^{-1} N$ for all $t \geq 0$. By inspection
\[
  \vertic_0 \Sigma^{-2}N = \Sigma^{-1} \Omega N 
\rightarrow 
\Sigma^{-1} \vertic_0 \Sigma^{-1} N = \Sigma^{-2}N 
\]
 is the natural inclusion and on $\vertic_1$ one has the identity.

\begin{rem}
 The above approach may be applied to 
consider modules of the form $M = \Sigma^{-d} N$ for $N \in \unst$ and 
any $d \in \nat$. The case $d=3$ already shows how the complexity in calculating 
the $E^2$-page
increases significantly with $d$.   
\end{rem}

\subsection{The spectral sequence for $\Sigma^n \cala^*$, with $n \in \nat$}
\label{subsect:Sigma_n_cala}

In \cite{Miller}, Miller explains (working cohomologically) how to recover 
$\cala/ \cala Sq^1$ from $\untor_* (\field, \field)$, the 
homological filtration corresponding to the length filtration of $\cala/ \cala 
Sq^1$ (see \cite[Example 3.3.12]{Miller}).

Similarly, the spectral sequence of Theorem \ref{thm:ss} can be used to 
calculate 
$\Sigma^n \cala^*$ (now working in homology). The $E^2$-page is concentrated on 
the anti-diagonal, 
hence the spectral sequence degenerates and 
\[
 E^\infty_{-t,t} = \stein_t \Sigma^{-1} F (n+t+1).
\]
Here, as in \cite{Miller}, the calculation is only carried out at the  level of 
graded vector spaces, although 
this can be made much more precise.

Recall the definition of an admissible sequence from Definition 
\ref{def:sequence_admissibility} and the fact that the excess of an admissible 
sequence identifies as 
$$
e(I) = i_1 - (\sum_{j>1} i_j).
$$
By convention, the excess of $\emptyset$ is taken to be $\infty$.

\begin{nota}
 For $I$ a sequence of length $s$ and $0 \leq j \leq  s$, write 
 \[
  \omega_j I := I \backslash\{ i_1, \ldots , i_{j} \},
 \]
so that $\omega_0 I = I$, $\omega_s I= \emptyset$ and, for $0< j <s$, 
$\omega_{j-1} I = \{i_j \} \cup 
\omega_{j}I$.
\end{nota}

Clearly if $I$ is admissible, so is $\omega_j I$ and $e (\omega_j I) \leq e 
(I)$ for $j <s$. 

\begin{prop}
\label{prop:unicity-gamma}
Let $d \in \zed$.  For $I$ an admissible sequence of length $s$, there is a 
unique integer $j \in \{0, \ldots , s\}$ such that both the following 
conditions hold
 \[
 \begin{array}{ll}
  e(\omega_j I) \leq d+j & \mathrm{if \ } j< s \\
  e(\omega_{j-1}I) \geq d+j &  \mathrm{if \ } j >0.
 \end{array}
\]
\end{prop}

\begin{proof}
 The definition of excess implies that, for $I$ admissible of length $s>0$, 
 \[
  e (I) \geq e(\omega_1 I) \geq e(\omega_2 I) \geq e(\omega_{s-1} I) >0.
 \]
Choose $j \in \{0, \ldots , s+1\}$ maximal such that the conditions of the 
Proposition hold. If $j=0$, there is nothing to prove. 
Otherwise, by the above, for $0 \leq k < j$ one has 
$$e(\omega_k I) \geq e(\omega_{j-1} I) \geq d+j > d+k.$$
Hence there is no solution to the given inequalities for $k <j$.  
\end{proof}

\begin{rem}
 For $j=0$, the condition of Proposition \ref{prop:unicity-gamma} reduces to $e 
(I) \leq d$; for  $j=s$ it reduces to $i_s \geq  d+ s$.  
\end{rem}

\begin{cor}
 \label{cor:unique_decomposition}
For $d \in \zed$ and $I \not = \emptyset$ an admissible sequence of length $s$, 
there 
is a unique $l \in \{0 , \ldots , s  \}$ such that 
\[
 \begin{array}{ll}
  e(\omega_j I) \leq d + j  & \mathrm{if\ } j <s \\
 i_j \geq  d+ d(\omega_j I) + j & \mathrm{if \ } j >0. 
 \end{array}
\]
\end{cor}

\begin{proof}
 Suppose that $j >0$.  Since $\omega_{j-1} I = \{i_j \} \cup \omega_j (I)$, one 
has 
 $e(\omega _{j-1}I) = i_j - d(\omega_j I)$, where $d(\omega_s I)=0$, by 
convention.
 Hence the condition $i_j \geq  d + d(\omega_j I) + j$ is equivalent to 
$e(\omega _{j-1}I)\geq  d+j$. The result therefore follows from  Proposition 
\ref{prop:unicity-gamma}.
\end{proof}

Corollary \ref{cor:unique_decomposition} is the combinatorial input to 
the spectral sequence calculating $\Sigma^n \cala^*$, together with 
the  calculation  of $\stein_* \Sigma^d \field$ given in 
Proposition \ref{prop:admissible_basis} and the following classical fact (see 
\cite{schwartz_book} for example):

\begin{prop}
\label{prop:basis-F(n)}
The free (cohomological) unstable module $F(n)$ on a generator of degree $n$ 
has basis $Sq^I \iota_n$, where $|\iota_n|= n$ and $I$ is an admissible 
sequence 
with $e(I) \leq n$.
\end{prop}

\begin{prop}
\label{prop:calc_n_lderq_F}
 For $0< n \in \nat$ and $t\in \nat$,  as graded vector spaces
 \[
 \stein_t  \Sigma^{-1} F (n+t+1) 
 \cong 
 \langle \sigma_{I'} Sq^{I''} \iota _n \rangle
 \]
where the sum ranges over admissible sequences $I'$, $I''$ such that
\begin{enumerate}
 \item 
 $l(I') = t$;
 \item 
  $e (I'') \leq (n+1)+t$;
  \item 
  $i'_t \geq (n+1) + d(I'') + t $ if $t>0$.
\end{enumerate}
The element $\sigma_{I'} Sq^{I''} \iota _n$ has degree $d(I') + d(I'') + n$. 
\end{prop}

\begin{proof}
 A basis of $\Sigma^{-1} F(n+t+1)$ is indexed by elements $\Sigma^{-1} Sq^{I''} 
\iota_{n+t+1}$, where $I''$ is admissible with $e (I'') \leq n+1+t$
  and the element has degree $d(I'') + n + t$. Proposition 
\ref{prop:admissible_basis} then gives the stated basis for 
  $\stein_t \Sigma^{-1} F (n+t+1)$.
\end{proof}

Combining Corollary \ref{cor:unique_decomposition} with Proposition 
\ref{prop:calc_n_lderq_F} shows that the spectral sequence calculates $\Sigma^n 
\cala^*$, as expected.

\begin{rem}
\ 
\begin{enumerate}
\item 
It is instructive to consider the suspension morphism of Theorem 
\ref{thm:ss_susp} in this case. 
Details are left to the reader.
 \item 
 This analysis extends to all integers $n \in \zed$. (Cf. Example 
\ref{exam:Sigma_n_cala}.)
\end{enumerate}
\end{rem}

\begin{rem} 
The infinite delooping spectral sequence of \cite{Miller} can also be used to 
recover $\Sigma^n \cala$, for $n 
\in \nat$; this relies upon an analysis of $\untor_* (\field, \Sigma F(n-1))$. 

 It is interesting to compare this with the above calculation. For this, 
Corollary \ref{cor:unique_decomposition} is replaced by the following 
observation for a fixed natural number $n \in \nat$:

 For an admissible sequence $I \not = \emptyset$ of length $s$, there is a 
unique integer $j \in \{0 , \ldots , s \}$ such that 
 the following conditions are satisfied:
 \[
  \begin{array}{ll}
  e(\omega_j I) \leq n& \mathrm{if \ } j<s \\
  i_j > d (\omega_j I ) +n & \mathrm{if \ } j> 0,
  \end{array}
 \]
where $d(\omega_s I) =0$, by convention. 

There is a degree shift which occurs 
in the calculation, due to the intervention of the homological degree of 
$\untor_s$. 
 \end{rem}

 \begin{rem}
  The above arguments based upon decompositions of admissible sequences are 
related to the calculations occurring in \cite[Section 5]{MH}, 
   which involve allowable monomial bases of the Dyer-Lashof algebra.
 \end{rem}


\providecommand{\bysame}{\leavevmode\hbox to3em{\hrulefill}\thinspace}
\providecommand{\MR}{\relax\ifhmode\unskip\space\fi MR }
\providecommand{\MRhref}[2]{%
  \href{http://www.ams.org/mathscinet-getitem?mr=#1}{#2}
}
\providecommand{\href}[2]{#2}

\end{document}